\documentclass[10pt]{article}

\usepackage{geometry}

\usepackage{parskip}

\usepackage{amsthm}
\usepackage{amsmath}
\usepackage{amsfonts}
\usepackage{amssymb}
\usepackage{tikz}
\usetikzlibrary{calc}

\allowdisplaybreaks

\usepackage{graphicx,subcaption}
\graphicspath{{figures/}}
\usepackage{hyperref}
\usepackage{color}
\newtheorem{definition}{Definition}[section]

\newtheorem{scaling}{Scaling assumption}
\newtheorem{theorem}[definition]{Theorem}
\newtheorem{remark}[definition]{Remark}
\newtheorem{corollary}[definition]{Corollary}

\newtheorem{proposition}[definition]{Proposition}

\newcommand{\R}{\mathbb{R}}
\newcommand{\G}{\mathcal{G}}
\newcommand{\C}{\mathcal{C}}
\newcommand{\vecsym}[1]{\boldsymbol{#1}}
\renewcommand{\vec}[1]{\mathbf{#1}}
\newcommand\doverline[1]{%
	\tikz[baseline=(nodeAnchor.base)]{
		\node[inner sep=0] (nodeAnchor) {$#1$}; 
		\draw[line width=0.1ex,line cap=round] 
		($(nodeAnchor.north west)+(0.0em,0.25ex)$) 
		--
		($(nodeAnchor.north east)+(0.0em,0.25ex)$) 
		($(nodeAnchor.north west)+(0.0em,0.55ex)$) 
		--
		($(nodeAnchor.north east)+(0.0em,0.55ex)$) 
		;
}}

\title{Continuous Limits for Constrained Ensemble Kalman Filter}
\date{\today}
\author{
	Michael Herty  and  Giuseppe Visconti \medskip\\
	{\small\it Institut f\"{u}r Geometrie und Praktische Mathematik (IGPM)} \\
	{\small\it RWTH Aachen University} \\
	{\small\it Templergraben 55, 52062 Aachen, Germany
	}
}

\begin{document}

\maketitle

\begin{abstract}
	The Ensemble Kalman Filter method can be used as an iterative particle numerical scheme for state dynamics estimation and control--to--observable identification problems. In applications it may be required to enforce the solution to satisfy equality constraints on the control space. In this work we deal with this problem from a constrained optimization point of view, deriving corresponding optimality conditions. Continuous limits, in time and in the number of particles, allows us to study properties of the method. We illustrate the performance of the method by using test inverse problems from the literature.
\end{abstract}

\paragraph{Mathematics Subject Classification (2010)} 65N21 (Inverse problems), 93E11 (Filtering), 37N35 (Dynamical systems in control), 35Q93 (PDEs in connection with control and optimization) %\textcolor{red}{--- 35Q93 (PDEs in connection with control and optimization), 37N35 (Dynamical systems in control), 37N40 (Dynamical systems in optimization and economics), 46N10 (Applications in optimization, convex analysis, mathematicalprogramming, economics)}

\paragraph{Keywords} Inverse problems, constrained optimization, nonlinear filtering methods, DAE systems, mean--field limit

\section{Introduction} \label{sec:introduction}

We are concerned with the following abstract inverse problem or parameter identification problem
\begin{equation} \label{eq:noisyProb}
\vec{y} = \G(\vec{u}) + \vecsym{\eta}
\end{equation}
where $\G:X \to Y$ is the (possible nonlinear) forward operator between the finite dimensional Hilbert spaces $X=\R^d$ and $Y=\R^K$, with $d,K\in\mathbb{N}$, $\vec{u}\in X$ is the control, $\vec{y}\in Y$ is the observation and $\vecsym{\eta}$ is  observational noise. Typically, $\vecsym{\eta}$ is not explicitly known but only information on its distribution is available. Throughout the paper we assume that $\vecsym{\eta} \sim \mathcal{N}(\vec{0},\vecsym{\Gamma}^{-1})$, i.e.~the observational noise is normally distributed with zero mean and given covariance matrix $\vecsym{\Gamma}^{-1}\in\R^{K\times K}$. Given noisy measurements or observations $\vec{y}$ and the known mathematical model $\G$, we are interested in finding the corresponding control $\vec{u}$. Inverse problems, in particular in view of a possible
ill--posedness, have been discussed in vast amount of literature and we refer to~\cite{EnglHankeNeubauer1996} for an introduction and further references.

We are interested in particle methods for solving numerically~\eqref{eq:noisyProb}. In particular, in the following we will investigate the Ensemble Kalman Filter (EnKF). While this method has already been introduced more than ten years ago~\cite{Evensen1994}, recent theoretical 
progress have been done both for the continuous time limit~\cite{schillingsetal2018,SchillingsPreprint,ChadaStuartTong2019,schillingsstuart2017,schillingsstuart2018} and the mean--field limit on the number of the ensemble members~\cite{CarrilloVaes,DingLi2019,Stuart2019MFEnKF,HertyVisconti2019}.
%In order to set up the mathematical formulation of the EnKF we consider to have the finite dimensional case $X=\R^d$ and $Y=\R^K$, with $d,K\in\mathbb{N}$.

Solving inverse problems or identification problems arising in realistic applications usually requires to consider constraints on the unknown control, as well as on the data. This has been demonstrated in the recent literature in several research fields, as for instance in weather forecasting~\cite{McLaughlin2014}, milling process~\cite{SchwenzerViscontiEtAl2019} and process control~\cite{Teixeira2010}. Including constraints in Kalman filtering is usually done via optimization techniques, see e.g.~\cite{aravkin2014optimization}. Recently, algorithms for the ensemble Kalman filter for constrained problems have been studied in~\cite{Stuart2019CEnKF,ChadaShillingsWeissmann2019. In particular, in~\cite{Stuart2019CEnKF}} the authors propose to solve the constrained compromise step for all the ensemble members that do not fulfill the constraints. However, we notice that no optimality conditions have been formally derived and that the solution of a large number of constrained optimization problems may be required. Instead, in~\cite{ChadaShillingsWeissmann2019} the authors focus on linear box--constraints and, via a projection--based method, derive a new variant of the EnKF and study the continuous time limit.

Here, we follow a similar approach to~\cite{Stuart2019CEnKF} by incorporating equality constraints in the compromise or update step of the unconstrained EnKF. Using the Lagrange theory for optimization problems, we derive first order necessary optimality conditions and discuss the link to game theory. The formulation of the optimality conditions shows that the unconstrained EnKF automatically satisfies linear equality constraints. Our analysis is therefore mainly focused to the case of nonlinear equality constraints. Under suitable scaling assumptions, we then compute the corresponding continuous time limit of the optimality conditions, which leads to a system of differential algebraic equations (DAEs). This reformulation allows us to perform an analysis of the method. A further continuous limit is analyzed in the mean--field limit, i.e.~in the regime of infinitely many ensembles.

The main contributions can be summarized in the following points:
\begin{itemize}
	\item We explicitly derive the ensemble Kalman update for inverse problems which automatically encapsulates and satisfies general nonlinear equality constraints. In this way, we avoid the possible solution of many constrained optimization problems for those ensemble members which do not satisfy the constraints.
	\item We compute and analyze the corresponding continuous limits, in time and in the number of the ensemble members, of the constrained ensemble Kalman update formula.
\end{itemize}

The paper is organized as follows. In Section~\ref{sec:constrainedEnKF} we briefly review the derivation of the EnKF and then we compute first order necessary optimality conditions for the constrained optimization problem, assuming only equality constraints. In Section~\ref{sec:continuousLimits} we study continuous limits, in time and in the number of ensembles, of the optimality conditions, providing an analysis for the resulting system of DAEs and mean--field equation. In Section~\ref{sec:numerics} we investigate the ability of the method to provide solution to an inverse problem with two types of equality constraints. Finally, we summarize the results in Section~\ref{sec:conclusion}. We report in Appendix~\ref{app:proof} detailed elaboration of computations in the proof of Proposition~\ref{th:optimality}. Appendix~\ref{app:symbols} is aimed to list the most used mathematical symbols throughout the paper.

\section{Ensemble Kalman filter for constrained problems} \label{sec:constrainedEnKF}

A particular numerical method for solving~\eqref{eq:noisyProb} is the Ensemble Kalman Filter (EnKF), which has been originally introduced to estimate state variables, parameters, etc.~of stochastic dynamical systems. The estimations are based on system dynamics and measurement data that are possibly perturbed by known noise. Therefore, in order to apply the EnKF to the inverse problem~\eqref{eq:noisyProb}, this is usually rewritten as a partially observed and artificial dynamical system based
on state augmentation, e.g.~cf.~\cite{Stuart2019CEnKF,Anderson2001,iglesiaslawstuart2013}.

Let us introduce the new variable $\vec{w} = \G(\vec{u})$, so that~\eqref{eq:noisyProb} can be reformulated as
\begin{align*}
\vec{w} &= \G(\vec{u})\\
\vec{y} &= \vec{w} + \vecsym{\eta}.
\end{align*}
Taking $\vec{y}^{n+1} = \vec{y}$ and $\vecsym{\eta}^{n+1} = \vecsym{\eta}$ as the given data and the given noise, respectively, we obtain the following dynamical system:
\begin{equation} \label{eq:dynSysExtended}
\begin{aligned}
\vec{u}^{n+1} &= \vec{u}^n\\
\vec{w}^{n+1} &= \G(\vec{u}^n)\\
\vec{y}^{n+1} &= \vec{w}^{n+1} + \vecsym{\eta}^{n+1}.
\end{aligned}
\end{equation}

\begin{remark}
	We stress the fact that here $n$ is an artificial time index, while in the case of a dynamic inverse problem physical time is included in $\G$. Moreover, in the following we keep the notation $\vec{y}^{n+1}$ even if we consider the free noise case where no perturbation in time is added to the initial noisy observation.
\end{remark}

In the following we briefly recall the derivation of the EnKF.
We use a compact formulation of the artificial dynamic process by defining $\vec{v} = (\vec{u},\vec{w})^\intercal \in \R^{d+K}$, $\vecsym{\Xi}(\vec{v}) = (\vec{u},\G(\vec{u}))^\intercal \in \R^{d+K}$ and the observational matrices $\vec{H} = [\vec{0},\vec{I}] \in \R^{K\times(d+K)}$, $\vec{H}^\perp=[\vec{I},\vec{0}] \in \R^{d\times(d+K)}$. We note that $\vec{H}\vec{v} = \vec{w}$ and $\vec{H}^\perp\vec{v} = \vec{u}$. Then, we rewrite~\eqref{eq:dynSysExtended} in the typical setting where the EnKF is applied for the solution of the inverse problem~\eqref{eq:noisyProb}:
\begin{equation} \label{eq:dynSys}
\begin{aligned}
\vec{v}^{n+1} &= \vecsym{\Xi}(\vec{v}^n)\\
\vec{y}^{n+1} &= \vec{H} \vec{v}^{n+1} + \vecsym{\eta}^{n+1}.
\end{aligned}
\end{equation}

Let us introduce $\{ \vec{v}^{j,n} \}_{j=1}^{J}$ the $J$ particles (ensembles) at time $n$. The objective of the EnKF is to reach a compromise $\vec{v}^{j,n+1}$ between the background estimate $\hat{\vec{v}}^{j,n+1}$ of the model and additional information provided by data $\vec{y}^{n+1}$, for each ensemble member. The state of the particles at time $n+1$ is predicted using the dynamics model~\eqref{eq:dynSys} to obtain
\begin{equation} \label{eq:prediction}
\hat{\vec{v}}^{j,n+1} = \vecsym{\Xi}(\vec{v}^{j,n}), \quad j=1,\dots,J.
\end{equation}
Let $\vec{C}^{n+1} \in \R^{(d+K)\times(d+K)}$ be the empirical covariance matrix characterizing the uncertainties on predictions $\hat{\vec{v}}^{j,n+1}$. We have
$$
\vec{C}^{n+1} = \frac{1}{J} \sum_{k=1}^J (\hat{\vec{v}}^{j,n+1} - \overline{\hat{\vec{v}}}^{n+1}) \otimes (\hat{\vec{v}}^{j,n+1} - \overline{\hat{\vec{v}}}^{n+1}), \quad
\overline{\hat{\vec{v}}}^{n+1} = \frac{1}{J} \sum_{k=1}^J \hat{\vec{v}}^{j,n+1}.
$$
It is also easy to check that
\begin{equation} \label{eq:covMatrix}
	\vec{C}^{n+1} = \begin{bmatrix} \vec{C}^{n+1}_{\vec{u}\vec{u}} & \vec{C}^{n+1}_{\vec{u}\vec{w}} \\ \vec{C}^{{n+1}^\intercal}_{\vec{u}\vec{w}} & \vec{C}^{n+1}_{\vec{w}\vec{w}}\end{bmatrix}
\end{equation}
where, using the definition of $\vecsym{\Xi}$,
\begin{gather*}
	\vec{C}^{n+1}_{\vec{u}\vec{u}} = \frac{1}{J} \sum_{k=1}^J (\vec{u}^{j,n} - \overline{\vec{u}}^{n}) \otimes (\vec{u}^{j,n} - \overline{\vec{u}}^{n}), \quad
	\vec{C}^{n+1}_{\vec{u}\vec{w}} = \frac{1}{J} \sum_{k=1}^J (\vec{u}^{j,n} - \overline{\vec{u}}^{n}) \otimes (\G(\vec{u}^{j,n}) - \overline{\G}^{n}) \\
	\vec{C}^{n+1}_{\vec{w}\vec{w}} = \frac{1}{J} \sum_{k=1}^J (\G(\vec{u}^{j,n}) - \overline{\G}^{n}) \otimes (\G(\vec{u}^{j,n}) - \overline{\G}^{n})
\end{gather*}
are the empirical covariance matrices depending on the ensemble set $\{\vec{u}^{j,n}\}_{j=1}^J$ at iteration $n$ and on $\{\G(\vec{u}^{j,n})\}_{j=1}^J$, i.e.~the image of $\{\vec{u}^{j,n}\}_{j=1}^J$ at iteration $n$, and where we define by $\overline{\vec{u}}^n$ and $\overline{\G}^n$ the mean of $\{\vec{u}^{j,n}\}_{j=1}^J$ and $\{\G(\vec{u}^{j,n})\}_{j=1}^J$, namely
$$
	\overline{\vec{u}}^n = \frac{1}{J} \sum_{j=1}^J \vec{u}^{j,n}, \quad \overline{\G}^n = \frac{1}{J} \sum_{j=1}^J \G(\vec{u}^{j,n}).
$$

\begin{remark} \label{rm:indexCNotation}
	The empirical covariance matrix of the ensembles is in fact computed by information at iteration step $n$. However, we keep the notation $\vec{C}^{n+1}_{\vec{u}\vec{u}}$.
\end{remark}

The compromise that is sought should have $\vec{v}^{j,n+1}$ close to $\hat{\vec{v}}^{j,n+1}$ and $\vec{H}\vec{v}^{j,n+1}$ close to $\vec{y}^{n+1}$. Then $\vec{v}^{j,n+1}$ is defined by the solution of the following minimization problem:
\begin{equation} \label{eq:compromise}
\vec{v}^{j,n+1} = \arg\min_{\vec{v}} \mathcal{J}^{j,n}(\vec{v}) := \frac12 \left\| \vec{y}^{n+1} - \vec{H}\vec{v} \right\|^2_{\vecsym{\Gamma}^{-1}} + \frac12 \left\| \vec{v} - \hat{\vec{v}}^{j,n+1} \right\|^2_{\vec{C}^{n+1}}.
\end{equation}
Here, we recall that $\vec{\Gamma}^{-1} \in \R^{K\times K}$ is the covariance matrix characterizing the uncertainties on data $\vec{y}^{n+1}$ and that we assume $\vecsym{\eta} \sim \mathcal{N}(\vec{0},\vecsym{\Gamma}^{-1})$. Throughout the work, we also assume that $\vecsym{\Gamma}^{-1}$ is positive definite.
We notice that the first term of $\mathcal{J}^{j,n}(\vec{v})$ corresponds to the least squares functional $\Phi$ given by
\begin{equation} \label{eq:leastSqFnc}
\Phi(\vec{u},\vec{y}) := \frac12 \left\| \vecsym{\Gamma}^{\frac12} (\vec{y} - \G(\vec{u})) \right\|^2.
\end{equation}
Further, $\mathcal{J}^{j,n}(\vec{v})$ can be written as
\begin{equation} \label{eq:compromiseEquivalent}
\mathcal{J}^{j,n}(\vec{v}) = \frac12 \vec{v}^\intercal (\vec{H}^\intercal \vecsym{\Gamma} \vec{H} + \vec{C}^{{n+1}^{-1}}) \vec{v} - (\vec{C}^{{n+1}^{-\intercal}} \hat{\vec{v}}^{j,n+1} + \vec{H}^\intercal \vecsym{\Gamma}^\intercal \vec{y}^{n+1} )^\intercal \vec{v} + \overline{\mathcal{J}}
\end{equation}
where $\overline{\mathcal{J}}$ collects all the terms independent of $\vec{v}$. The EnKF update formula is derived by imposing first order necessary condition $\nabla_{\vec{v}} \mathcal{J}^{j,n}(\vec{v}) = \vec{0}$. For an extensive discussion, we refer e.g.~to~\cite{Stuart2019CEnKF,iglesiaslawstuart2013}. In particular, we point--out here that the derivation of the EnKF update formula through~\eqref{eq:compromiseEquivalent} requires that the empirical covariance $\vec{C}^{n+1}$ is positive definite. This is a strong assumption which cannot be guaranteed in a general fashion and it is usually overcame by introducing a shifting of $\vec{C}^{n+1}$ as $\vec{C}_\varepsilon^{n+1} = \vec{C}^{n+1}+\varepsilon\vec{I}$, where $\vec{I}\in\R^{(d+K)\times(d+K)}$ is the identity matrix, and study the limit of $\varepsilon \to 0$. We refer to~\cite{Stuart2019CEnKF}. For the derivation of the EnKF update formula for constrained problems we will follow this approach, see Proposition~\ref{th:optimality} in the following section.

We recall here that in order to help the reader, we report a list of the main mathematical symbols in Appendix~\ref{app:symbols}.

\subsection{Equality constraints on the control space} \label{sec:controlLinConstr}

We consider imposing equality constraints in the space of the control and formulate those as
\begin{equation} \label{eq:constraint}
\mathcal{A}(\vec{u}) = \vec{0}_{\R^m}
\end{equation}
where $\mathcal{A}$ is the vector valued differentiable operator $\mathcal{A} : \vec{u} \in \R^d \mapsto \mathcal{A}(\vec{u}) \in \R^m$, containing the $m \leq d$ constraint values.
\begin{remark}
	Throughout the paper, we use the subscript $(\ell)$, $\ell=1,\dots,N$, to denote the $\ell$--th component of vectors and vector--valued functions on $\R^N$.
\end{remark}
Let $\vec{J}_\mathcal{A} = \left[ \nabla_{\vec{u}} \mathcal{A}_{(1)},\dots,\nabla_{\vec{u}} \mathcal{A}_{(m)} \right]^\intercal$ be the $m \times d$ Jacobian matrix of the operator $\mathcal{A}$. In order to satisfy constraint qualification for equality constraints, we require that if $\vec{u}^*$ is a feasible point then $\vec{J}_\mathcal{A}(\vec{u}^*)$ has full rank, i.e.~$m$. In other words, $\vec{u}^*$ is a regular point of the constraint.

\begin{theorem}[see~\cite{FiaccoMcCormick1968,KuhnTucker1951}]
	If $\vec{u}^*$ is a feasible point and $\vec{J}_\mathcal{A}(\vec{u}^*)$ is a full rank matrix, the constraint qualification LICQ holds at $\vec{u}^*$. Hence, the Lagrange multipliers theorem gives necessary optimality condition.
\end{theorem}

In order to formulate the constrained optimization problem associated to the minimization of~\eqref{eq:compromiseEquivalent}, we define also the extension of the operator $\mathcal{A}$ to the space $\R^{d+K}$ as $\tilde{\mathcal{A}} : \vec{v}\in\R^{d+K} \mapsto \tilde{\mathcal{A}}(\vec{v}) = \mathcal{A}(\vec{H}^\perp \vec{v}) \in \R^m$. Then observe that the Jacobian $\vec{J}_{\tilde{\mathcal{A}}} \in \R^{m\times(d+K)}$ of the operator $\tilde{\mathcal{A}}$ is $\vec{J}_{\tilde{\mathcal{A}}} = [\vec{J}_\mathcal{A},\vec{0}]$ with $\vec{0}\in\R^{m\times K}$ since $\partial_{\vec{v}_{(\ell)}} \mathcal{A}_{(k)} = 0$, for $\ell=d+1,\dots,d+K$ and $k=1,\dots,m$. Moreover, if $\vec{u}^*$ is feasible and $\vec{J}_\mathcal{A}(\vec{u}^*)$ has rank $m$, then $\vec{J}_{\tilde{\mathcal{A}}}(\vec{u}^*)$ has also rank $m$.

The optimization step becomes then
\begin{equation} \label{eq:constrainedProblem}
\min_{\vec{v}} \mathcal{J}^{j,n}(\vec{v}), \quad \text{subject to } \tilde{\mathcal{A}}(\vec{v})=\vec{0}_{\R^m}, \quad j=1,\dots,J.
\end{equation}
It is clear that~\eqref{eq:constrainedProblem} requires to solve $J$ constrained optimization problems sequentially, at each iteration step $n$. We point out that the objective functional~\eqref{eq:compromiseEquivalent} is convex.

\begin{proposition} \label{th:optimality}
	Let $\vec{v}^{j,n+1}$ be an optimal solution to the constrained optimization problem~\eqref{eq:constrainedProblem} for a given $j$ and $n$, satisfying the constraint qualification for the differentiable equality constraint defined by $\tilde{\mathcal{A}}$. Then $\vec{u}^{j,n+1} = \vec{H}^\perp \vec{v}^{j,n+1}$ satisfies the first order necessary optimality conditions
	\begin{equation} \label{eq:updateControl}
	\begin{aligned}
	\vec{u}^{j,n+1} =& \vec{u}^{j,n} + \vec{C}^{n+1}_{\vec{u}\vec{w}} ( \vec{C}^{n+1}_{\vec{w}\vec{w}} + \vecsym{\Gamma}^{-1} )^{-1} (\vec{y}^{n+1} - \G(\vec{u}^{j,n}) )\\
		&+ \vec{C}^{n+1}_{\vec{u}\vec{w}} ( \vec{C}^{n+1}_{\vec{w}\vec{w}} + \vecsym{\Gamma}^{-1} )^{-1} \vec{C}^{{n+1}^\intercal}_{\vec{u}\vec{w}} \vec{J}_\mathcal{A}^\intercal(\vec{u}^{j,n+1}) \vecsym{\lambda}^{j,n+1} - \vec{C}^{n+1}_{\vec{u}\vec{u}} \vec{J}_\mathcal{A}^\intercal(\vec{u}^{j,n+1}) \vecsym{\lambda}^{j,n+1} \\
	\mathcal{A}(\vec{u}^{j,n+1}) =& \vec{0}_{\R^m}.
	\end{aligned}
	\end{equation}
\end{proposition}
\begin{proof}	
We apply Lagrange multiplier technique to the objective $\mathcal{J}_\varepsilon^{j,n}(\vec{v})$ defined by~\eqref{eq:compromiseEquivalent} with shifting of the empirical covariance $\vec{C}^{n+1} \mapsto \vec{C}_\varepsilon^{n+1}:=\vec{C}^{n+1}+\varepsilon\vec{I}$, where $\vec{I}\in\R^{(d+k)\times(d+K)}$ is the identity matrix and $\varepsilon>0$. Then we study the limit $\varepsilon \to 0$. We observe that, thanks to the shifted matrix, $\vec{H}^\intercal \vecsym{\Gamma} \vec{H} + \vec{C}_\varepsilon^{{n+1}^{-1}}$ is strictly positive definite. Consequently, the objective $\mathcal{J}_\varepsilon^{j,n}$ is strongly convex and there is existence and uniqueness of the solution to the minimization problem.

Detailed computations required in the following proof are given in Appendix~\ref{app:proof}.

The constrained minimum of $\mathcal{J}_\varepsilon^{j,n}(\vec{v})$ corresponds to a stationary point of the Lagrangian function
$$
\mathfrak{L}(\vec{v},\vecsym{\lambda}) = \mathcal{J}_\varepsilon^{j,n}(\vec{v}) + \vecsym{\lambda}^\intercal \tilde{\mathcal{A}}(\vec{v})
$$
where the vector $\vecsym{\lambda} \in \R^m$ contains the Lagrange multipliers.
Variations of the Lagrangian function $\mathfrak{L}$ with respect to $\vec{v}$ and $\vecsym{\lambda}$ are
\begin{align*}
\nabla_{\vec{v}} \mathfrak{L}(\vec{v},\vecsym{\lambda}) &=  (\vec{H}^\intercal \vecsym{\Gamma} \vec{H} + \vec{C}_\varepsilon^{{n+1}^{-1}}) \vec{v} - (\vec{H}^\intercal \vecsym{\Gamma}^\intercal \vec{y}^{n+1} + \vec{C}_\varepsilon^{{n+1}^{-1}} \hat{\vec{v}}^{j,n+1}) + (\vecsym{\lambda}^\intercal \vec{J}_{\tilde{\mathcal{A}}}(\vec{v}))^\intercal \\
\nabla_{\vecsym{\lambda}} \mathfrak{L}(\vec{v},\vecsym{\lambda}) &= \tilde{\mathcal{A}}(\vec{v}).
\end{align*}
If $(\vec{v}^{j,n+1},\vecsym{\lambda}^{j,n+1})\in\R^{d+K}\times\R^m$ is an optimal solution then it satisfies $\nabla_{\vec{v}} \mathfrak{L}(\vec{v}^{j,n+1},\vecsym{\lambda}^{j,n+1}) = \vec{0}_{\R^{d+K}}$ and $\nabla_{\vecsym{\lambda}} \mathfrak{L}(\vec{v}^{j,n+1},\vecsym{\lambda}^{j,n+1}) = \vec{0}_{\R^m}$, that are the first order necessary optimality conditions.

Due to the assumption of a regular point, the system of optimality conditions is determined. In particular, imposing $\nabla_{\vec{v}} \mathfrak{L}(\vec{v}^{j,n+1},\vecsym{\lambda}^{j,n+1}) = \vec{0}_{\R^{d+K}}$, it is possible to determine
\begin{equation} \label{eq:vSol}
\vec{v}^{j,n+1} = (\vec{H}^\intercal \vecsym{\Gamma} \vec{H} + \vec{C}_\varepsilon^{{n+1}^{-1}})^{-1} (\vec{H}^\intercal \vecsym{\Gamma} \vec{y}^{n+1} + \vec{C}_\varepsilon^{{n+1}^{-1}} \hat{\vec{v}}^{j,n+1}) - (\vec{H}^\intercal \vecsym{\Gamma} \vec{H} + \vec{C}_\varepsilon^{{n+1}^{-1}})^{-1} \vec{J}_{\tilde{\mathcal{A}}}^\intercal(\vec{v}^{j,n+1}) \vecsym{\lambda}^{j,n+1}.
\end{equation}
Now we apply the Woodbury matrix identity to the term multiplying $\hat{\vec{v}}^{j,n+1}$ and to the last term in~\eqref{eq:vSol}, cf.~Appendix~\ref{app:proof}.
%:
%$$
%(\vec{M} + \vec{U} \vec{N} \vec{V})^{-1} = \vec{M}^{-1} - \vec{M}^{-1} \vec{U} (\vec{N}^{-1} + \vec{V} \vec{M}^{-1} \vec{U})^{-1} \vec{V} \vec{M}^{-1},
%$$
%for all matrices $\vec{M}$, $\vec{U}$, $\vec{N}$, $\vec{V}$.
Instead, for the term multiplying $\vec{y}^{n+1}$ in~\eqref{eq:vSol} we observe that
$$
(\vec{H}^\intercal \vecsym{\Gamma} \vec{H} + \vec{C}_\varepsilon^{{n+1}^{-1}})^{-1} \vec{H}^\intercal \vecsym{\Gamma} = \vec{C}_\varepsilon^{n+1} \vec{H}^\intercal (\vec{H} \vec{C}_\varepsilon^{n+1} \vec{H}^\intercal + \vecsym{\Gamma}^{-1})^{-1}.
$$
Then we get, cf.~Appendix~\ref{app:proof},
\begin{equation} \label{eq:vSol2}
\begin{aligned}
\vec{v}^{j,n+1} =& \hat{\vec{v}}^{j,n+1} + \vec{C}_\varepsilon^{n+1} \vec{H}^\intercal (\vec{H} \vec{C}_\varepsilon^{n+1} \vec{H}^\intercal + \vecsym{\Gamma}^{-1})^{-1} (\vec{y}^{n+1} - \vec{H}\hat{\vec{v}}^{j,n+1}) \\
&- (\vec{C}_\varepsilon^{n+1} - \vec{C}_\varepsilon^{n+1} \vec{H}^\intercal (\vec{H} \vec{C}_\varepsilon^{n+1} \vec{H}^\intercal + \vecsym{\Gamma}^{-1})^{-1} \vec{H} \vec{C}_\varepsilon^{n+1} ) \vec{J}_{\tilde{\mathcal{A}}}^\intercal(\vec{v}^{j,n+1}) \vecsym{\lambda}^{j,n+1}. %\\
%=& \hat{\vec{v}}^{j,n+1} + \vec{C}^{{n+1}^\intercal} \vec{H}^\intercal (\vec{H} \vec{C}^{n+1} \vec{H}^\intercal + \vecsym{\Gamma}^{-1})^{-1} (\vec{y}^{n+1} - \vec{H}\hat{\vec{v}}^{j,n+1} - \vec{H} \vec{C}^{n+1} \vec{H}^{\perp^\intercal} \vec{A}^\intercal \vecsym{\lambda}^{j,n+1}) \\
%&+ \vec{C}^{n+1} \vec{H}^{\perp^\intercal} \vec{A}^\intercal \vecsym{\lambda}^{j,n+1}.
\end{aligned}
\end{equation}
If we multiply the previous equation by the observation matrix $\vec{H}^\perp$ and compute the limit $\varepsilon\to 0$, we obtain the update formula for the control, cf.~Appendix~\ref{app:proof}:
%\begin{align*}
%\vec{u}^{j,n+1} = \vec{H}^\perp \vec{v}^{j,n+1} =& \vec{H}^\perp \hat{\vec{v}}^{j,n+1} + \vec{H}^\perp \vec{C}^{n+1} \vec{H}^\intercal (\vec{H} \vec{C}^{n+1} \vec{H}^\intercal + \vecsym{\Gamma}^{-1})^{-1} (\vec{y}^{n+1} - \vec{H}\hat{\vec{v}}^{j,n+1}) \\
%&- \vec{H}^\perp (\vec{C}^{n+1} - \vec{C}^{n+1} \vec{H}^\intercal (\vec{H} \vec{C}^{n+1} \vec{H}^\intercal + \vecsym{\Gamma}^{-1})^{-1} \vec{H} \vec{C}^{n+1} ) \vec{J}_{\tilde{\mathcal{A}}}^\intercal(\vec{v}^{j,n+1}) \vecsym{\lambda}^{j,n+1}.
%\end{align*}
%This can be further simplified by noticing that
%\begin{gather*}
%\vec{H}^\perp \hat{\vec{v}}^{j,n+1} = \vec{u}^{j,n}, \quad
%\vec{H}^\perp \vec{C}^{n+1} \vec{H}^\intercal = \vec{C}^{n+1}_{\vec{u}\vec{w}}, \quad %\\
%\vec{H} \vec{C}^{n+1} \vec{H}^\intercal = \vec{C}^{n+1}_{\vec{w}\vec{w}}, \quad \vec{H}\hat{\vec{v}}^{j,n+1} = \G(\vec{u}^{j,n}) \\
%\vec{H} \vec{C}^{n+1} \vec{J}_{\tilde{\mathcal{A}}}^\intercal(\vec{v}^{j,n+1}) = \vec{C}^{{n+1}^\intercal}_{\vec{u}\vec{w}} \vec{J}_{\mathcal{A}}^\intercal(\vec{u}^{j,n+1}), \quad \vec{H}^\perp \vec{C}^{n+1} \vec{J}_{\tilde{\mathcal{A}}}^\intercal(\vec{v}^{j,n+1}) = \vec{C}^{n+1}_{\vec{u}\vec{u}} \vec{J}_{\mathcal{A}}^\intercal(\vec{u}^{j,n+1}).
%\end{gather*}
%We finally obtain
\begin{align*}
\vec{u}^{j,n+1} =& \vec{u}^{j,n} + \vec{C}^{n+1}_{\vec{u}\vec{w}} ( \vec{C}^{n+1}_{\vec{w}\vec{w}} + \vecsym{\Gamma}^{-1} )^{-1} (\vec{y}^{n+1} - \G(\vec{u}^{j,n}) )\\
&+ \vec{C}^{n+1}_{\vec{u}\vec{w}} ( \vec{C}^{n+1}_{\vec{w}\vec{w}} + \vecsym{\Gamma}^{-1} )^{-1} \vec{C}^{{n+1}^\intercal}_{\vec{u}\vec{w}} \vec{J}_\mathcal{A}^\intercal(\vec{u}^{j,n+1}) \vecsym{\lambda}^{j,n+1} - \vec{C}^{n+1}_{\vec{u}\vec{u}} \vec{J}_\mathcal{A}^\intercal(\vec{u}^{j,n+1}) \vecsym{\lambda}^{j,n+1},
\end{align*}
coupled to $\mathcal{A}(\vec{u}^{j,n+1})=\vec{0}_{\R^m}$, which provide a set of necessary optimality conditions for the constrained optimization problem~\eqref{eq:constrainedProblem}.
\end{proof}

Observe that control update~\eqref{eq:updateControl} is the classical ensemble Kalman filter update formula, given by the first two terms in the right hand side, perturbed by the last two terms due to the constraint~\eqref{eq:constraint}. Moreover, \eqref{eq:updateControl} depends also on the multipliers and their values must be determined as part of the solution.

\begin{corollary} \label{th:convex}
	Consider the assumptions of Proposition~\ref{th:optimality}. If the feasible set is convex, then the optimization problem~\eqref{eq:constrainedProblem} is convex, and this implies that the necessary optimality conditions~\eqref{eq:updateControl} are also sufficient. In particular, this holds true for affine equality constraints.
\end{corollary}

\subsubsection{A Game Theory viewpoint and existence results}
	
	Proposition~\ref{th:optimality} is stated for a fixed ensemble member $j$ at a fixed iteration $n$, by assuming that the other ensembles play a role of parameters in the optimization problem. This is closely related to the fact that
	problem~\eqref{eq:constrainedProblem} can be seen from a game theoretic point of view. In fact, if we interpret each ensemble member $j \in \{ 1,\dots,J \}$ as a player which chooses a control variable $\vec{v}^{j,n+1}$ in its set of feasible controls given by the constraint $\tilde{\mathcal{A}}(\vec{v}^{j,n+1})=\vec{0}$ while seeking to minimize its payoff function $\mathcal{J}^{j,n}$ which depends, via the covariance matrix $\vec{C}^{n+1}$, on all other players' controls at iteration $n$, i.e.~$\vec{v}^{-j,n}$, cf.~Remark~\ref{rm:indexCNotation}. Here, we indicate $\vec{v}^{-j,n+1} = \{ \vec{v}^{k,n+1} : k=1,\dots,J, \ k\neq j \}$. Observe that each player has no knowledge on the strategy adopted by other players. Moreover, we notice that we are in presence of a repeated game.
	
	In a non--cooperative and simultaneous game, we recall that $\{ \vec{v}^{j,n+1} : j=1,\dots,J \}$ represents a Nash equilibrium if for every $\vec{v}$ satisfying the constraint one has $\mathcal{J}^{j,n}(\vec{v},\vec{v}^{-j,n+1}) \geq \mathcal{J}^{j,n}(\vec{v}^{j,n+1},\vec{v}^{-j,n+1})$, for $j=1,\dots,J$.
	
	In general, a Nash equilibrium may not exist and need not to be unique. Conditions for the existence of Nash equilibrium are recalled in the following theorem.
	\begin{theorem}[Nikaid\^{o} and Isoda~\cite{NikaidoIsoda1955}]
		Let $\Gamma = \{ X_i,f_i \}_{i=1}^N$ be a game with nonempty, compact and convex strategy sets $X_i \subset \R^{n_i}$ and continuous payoff functions $f_i : X_i \to \R$, which are quasiconvex in $x^i$ for every fixed $x^{-i}$ for all $i=1,\dots,N$. Then $\Gamma$ has (at least) one Nash equilibrium.
	\end{theorem}
	We recall that quasiconvexity is implied by convexity. Then the Nikaid\^{o} and Isoda Theorem is satisfied for problem~\eqref{eq:constrainedProblem} if the feasible set is convex.
	In fact, the payoff function~\eqref{eq:compromiseEquivalent} is convex since $\vec{H}^\intercal \vec{\Gamma} \vec{H} + \vec{C}^{{n+1}^{-1}}$ is positive semidefinite, and consequently $\mathcal{J}^{j,n}$ is quasiconvex for each $j=1,\dots,J$. In particular, there exists at least one equilibria for affine equality constraints.
	For further references, we refer to~\cite{BressanNotes,FacchineiPang,SchwartzNotes}.

\subsection{The case of linear equality constraints} \label{sec:linearEqConstr}

Let us assume that the constraint $\mathcal{A}$ on the control space is linear. Then we can write $\mathcal{A}(\vec{u}) = \vec{A} \vec{u}$, with $\vec{A} \in \R^{m\times d}$. In this case we need to assume that $\vec{A} \in \R^{m \times d}$ is a full row rank matrix in order to have satisfied the condition of regular point. The linearity of the equality constraint guarantees an expression to evaluate the multipliers explicitly by substituting the update formula given in~\eqref{eq:updateControl} into the constraint~\eqref{eq:constraint}.
%$$
%\vec{A} \tilde{\vec{u}}^{j,n+1} + \vec{A} \vec{C}^{n+1}_{\vec{u}\vec{w}} (\vec{C}^{n+1}_{\vec{w}\vec{w}} + \vecsym{\Gamma}^{-1})^{-1} \vec{C}^{{n+1}^\intercal}_{\vec{u}\vec{w}} \vec{A}^\intercal \vecsym{\lambda}^{j,n+1} - \vec{A} \vec{C}^{n+1}_{\vec{u}\vec{u}} \vec{A}^\intercal \vecsym{\lambda}^{j,n+1} - \vec{b} = \vec{0}
%$$
%where $\tilde{\vec{u}}^{j,n+1}$ is the update for the unconstrained problem, namely
%$$
%\tilde{\vec{u}}^{j,n+1} = \vec{u}^{j,n} + \vec{C}^{n+1}_{\vec{u}\vec{w}} ( \vec{C}^{n+1}_{\vec{w}\vec{w}} + \vecsym{\Gamma}^{-1})^{-1} (\vec{y}^{n+1} - \G(\vec{u}^{j,n})).
%$$
%Using the fact that $\vec{A} \vec{u}^{j,n}-\vec{b}=\vec{0}$, one can easily get
%\begin{equation} \label{eq:multipliers}
%\left( \vec{A} \vec{C}_{\vec{u}\vec{u}}^{n+1} \vec{A}^\intercal - \vec{A} \vec{C}_{\vec{u}\vec{w}}^{n+1} \left( \vec{C}_{\vec{w}\vec{w}}^{n+1} + \vecsym{\Gamma}^{-1} \right)^{-1} \vec{C}_{\vec{u}\vec{w}}^{{n+1}^\intercal} \vec{A}^\intercal \right) \vecsym{\lambda}^{j,n+1} = \vec{A} \vec{C}_{\vec{u}\vec{w}}^{n+1} \left( \vec{C}_{\vec{w}\vec{w}}^{n+1} + \vecsym{\Gamma}^{-1} \right)^{-1} \left( \vec{y} - \G\left( \vec{u}_j^n \right) \right)
%\end{equation}
The following result holds true in the case of linear equality constraints.

\begin{proposition} \label{th:linearConstraint}
	Let $\mathcal{A}:\R^d\to\R^m$ be a linear operator and $\{ \vec{u}^{j,0} \}_{j=1}^J$ be a set of initial ensembles such that $\mathcal{A}(\vec{u}^{j,0}) = \vec{0}_{\R^m}$. Let $\{ \vec{u}^{j,n+1} \}_{j=1}^J$ be the ensemble set computed via the unconstrained ensemble Kalman filter update. Then $\mathcal{A}(\vec{u}^{j,n+1}) = \vec{0}_{\R^m}$, $\forall\,j=1,\dots,J$.
\end{proposition}
\begin{proof}
	We prove the statement by induction on $n$. Notice that $\vec{J}_\mathcal{A} = \vec{A}^\intercal$ since the constraint is linear. It is easy to check that if $\{ \vec{u}^{j,n} \}_{j=1}^J$ satisfy the constraint then $\vec{C}_{\vec{u}\vec{w}}^{{n+1}^\intercal} \vec{A}^\intercal = \vec{0}_{\R^{K\times m}}$ and $\vec{C}_{\vec{u}\vec{u}}^{{n+1}^\intercal} \vec{A}^\intercal = \vec{0}_{\R^{d\times m}}$. Then the update formula given in~\eqref{eq:updateControl} reduces to
	$$
	\vec{u}^{j,n+1} = \vec{u}^{j,n} + \vec{C}^{n+1}_{\vec{u}\vec{w}} ( \vec{C}^{n+1}_{\vec{w}\vec{w}} + \vecsym{\Gamma}^{-1})^{-1} (\vec{y}^{n+1} - \G(\vec{u}^{j,n}))
	$$
	which represents the unconstrained ensemble Kalman filter formula.
	Substituting into the constraint we have $\vec{A} \vec{u}^{j,n+1} = \vec{A} \vec{u}^{j,n} = \vec{0}_{\R^m}$ since also $\vec{A} \vec{C}_{\vec{u}\vec{w}}^{n+1} = \vec{0}_{\R^{m\times K}}$.
\end{proof}

We point--out that the previous result is a direct consequence of the well--known subspace
property of the classical unconstrained EnKF formula, which states that the ensemble iterates are a linear combination of the initial ensemble members, namely they stay in the subspace spanned by the initial ensemble~\cite{iglesiaslawstuart2013}.

%\subsubsection{Numerical experiment}
We provide an experimental evidence of the result in Proposition~\ref{th:linearConstraint}. We consider the inverse problem of finding the force function of a linear elliptic equation in one spatial dimension assuming that noisy observation of the solution to the problem are available. This is a standard problem in the mathematical literature on the EnKF for inverse problems, e.g.~see~\cite{Stuart2019CEnKF,iglesiaslawstuart2013,schillingsstuart2017}.

The problem is prescribed by the following one dimensional elliptic PDE
\begin{equation} \label{eq:ellipticEq}
-\frac{\mathrm{d}^2}{\mathrm{d}x^2} p(x) + p(x) = u(x), \quad x\in[0,\pi]
\end{equation}
endowed with boundary conditions $p(0) = p(\pi) = 0$. We assign a continuous exact control $u(x)$, being symmetric with respect to $x=\frac{\pi}{2}$, namely $u(x) = u(-x+\pi)$. Introducing a uniform mesh consisting of $d=K=2^8$ equidistant points on the interval $[0,\pi]$, we let $\vec{u}^\dagger\in\R^d$ be the vector of the evaluations of the control function $u(x)$ on the mesh. Noisy observations $\vec{y}\in\R^K$ are simulated as
$$
\vec{y} = \vec{p} + \vecsym{\eta} = \vec{G} \vec{u}^\dagger + \vecsym{\eta},
$$
where $\vec{G}\in\R^{K\times d}$ is the finite difference discretization of the continuous linear operator defining the elliptic PDE~\eqref{eq:ellipticEq}. For simplicity we assume that $\vecsym{\eta}$ is a Gaussian white noise, more precisely $\vecsym{\eta}\sim \mathcal{N}(0,\gamma^2 \vec{I})$ with $\gamma \in \R^+$ and $\vec{I} \in \R^{d\times d}$ is the identity matrix. We are interested in recovering the control $\vec{u}^\dagger \in \R^d$ from the noisy observations $\vec{y}\in\R^K$ only.

\begin{figure}[!t]
	\centering
	\includegraphics[width=\textwidth]{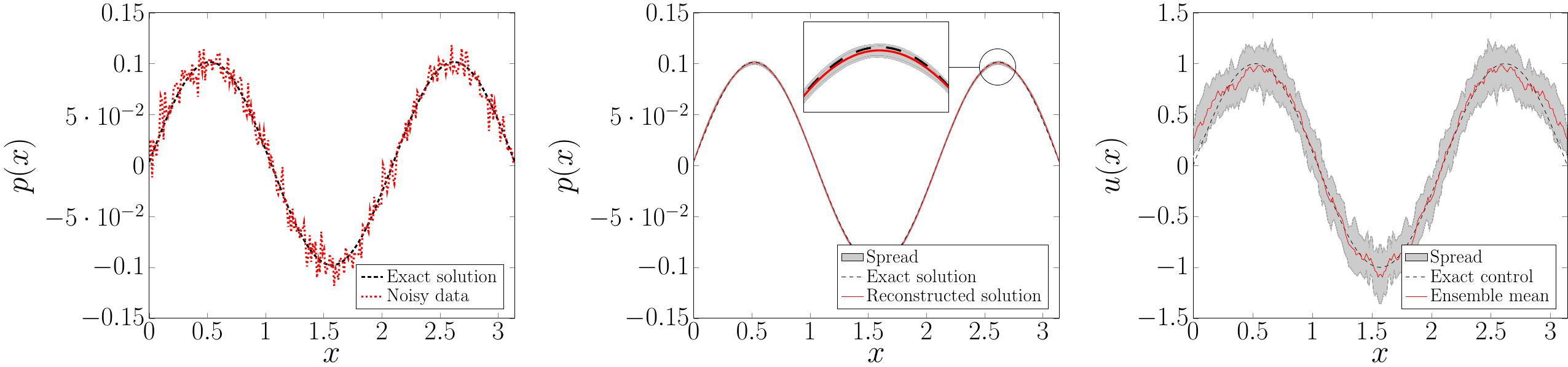}
	\caption{Application of the unconstrained ensemble Kalman filter based on the discrete formulation, assuming that the control function satisfies a linear equality constraint given as symmetry with respect to $x=\frac{\pi}{2}$.\label{fig:linearCEnKF}}
\end{figure}

Let us consider $u(x) = \sin(3x)$, $\forall\,x\in[0,\pi]$. In Figure~\ref{fig:linearCEnKF} we show the solution to this problem provided by the unconstrained ensemble Kalman filter. The initial ensemble set $\{\vec{u}^{j,0}\}_{j=1}^J$ is artificially built in order to satisfy symmetry with respect to $x=\frac{\pi}{2}$, after being sampled from a Brownian bridge, e.g.~as in~\cite{schillingsstuart2017}. This implies that $m=\frac{d}{2}=2^7$ constraints are taken into account. We solve the inverse problem by updating the ensemble members with
$$
	\vec{u}^{j,n+1} = \vec{u}^{j,n} + \vec{C}^{n+1}_{\vec{u}\vec{w}} ( \vec{C}^{n+1}_{\vec{w}\vec{w}} + \vecsym{\Gamma}^{-1})^{-1} (\vec{y} - G(\vec{u}^{j,n}))
$$
for $J=100$ and a noise level $\gamma=0.01$. The filter method converges, meeting the discrepancy principle in very few iterations, and allowing each ensemble member and the corresponding mean to satisfy the constraint.

Proposition~\ref{th:linearConstraint} guarantees that the constrained optimization problem~\eqref{eq:constrainedProblem} is solved by the unconstrained ensemble Kalman filter when the constraint $\mathcal{A}$ is linear, independently on the linearity of the model $\G$. As consequence, analysis and continuous limits of the ensemble Kalman filter with linear equality constraints coincide with the recent results~\cite{SchillingsPreprint,Stuart2019MFEnKF,HertyVisconti2019,schillingsstuart2017,schillingsstuart2018} for unconstrained problems, and in the following we will focus on nonlinear constraints only.

\section{Continuous limits of the constrained ensemble Kalman filter} \label{sec:continuousLimits}

In this section we investigate continuous limits of the constrained version of the ensemble Kalman filter, when nonlinear equality constraints are taken into account. These limits allow us to reformulate the fully discrete filter into the framework of ordinary and partial differential equations, thus making the method computationally simpler and amenable to an analysis of its properties.

\subsection{Continuous time limit} \label{sec:continuousTime}

In order to compute the continuous time limit equation, we make the following assumptions, see e.g.~\cite{HertyVisconti2019,schillingsstuart2017}.
\begin{scaling}
	The covariance matrix $\vecsym{\Gamma}^{-1}$, accounting for uncertainties due to data, is scaled by a scalar parameter $\Delta t$.
\end{scaling}
\begin{scaling}
	The multipliers $\vecsym{\lambda}^{j,n+1}$ are scaled as $\Delta t\vecsym{\lambda}^{j,n+1}$.
\end{scaling}
We notice that the second assumption is in fact not restrictive since multipliers are unique up to a multiplicative constant.

\begin{proposition}
	Under the scaling assumptions 1 and 2, the constrained ensemble Kalman filter~\eqref{eq:updateControl} formally converges to the semi--explicit system of differential algebraic equations (DAEs)
	\begin{equation} \label{eq:limitDAE}
	\begin{aligned}
	\frac{\mathrm{d}}{\mathrm{d} t} \vec{u}^j &= \vec{C}_{\vec{u}\vec{w}} \vecsym{\Gamma} (\vec{y} - \G(\vec{u}^j)) - \vec{C}_{\vec{u}\vec{u}} \vec{J}_\mathcal{A}^\intercal(\vec{u}^j) \vecsym{\lambda}^j \\
	\vec{0}_{\R^m} &= \mathcal{A}(\vec{u}^j(t)),
	\end{aligned}
	\end{equation}
	in the limit $\Delta t \to 0^+$.
\end{proposition}
\begin{proof}
Using the scaling assumptions, the solution~\eqref{eq:updateControl} to the constrained optimization problem can be reformulated as
\begin{equation} \label{eq:updateControlScaled}
\begin{aligned}
\vec{u}^{j,n+1} %=& \vec{u}^{j,n} + \Delta t \vec{C}^{n+1}_{\vec{u}\vec{w}} ( \Delta t \vec{C}^{n+1}_{\vec{w}\vec{w}} + \vecsym{\Gamma}^{-1})^{-1} (\vec{y}^{n+1} - \G(\vec{u}^{j,n})) \\
%&+ \Delta t^2 \vec{C}^{n+1}_{\vec{u}\vec{w}} ( \Delta t \vec{C}^{n+1}_{\vec{w}\vec{w}} + \vecsym{\Gamma}^{-1})^{-1} \vec{C}^{{n+1}^\intercal}_{\vec{u}\vec{w}} \vec{J}_\mathcal{A}^\intercal(\vec{u}^{j,n+1}) \vecsym{\lambda}^{j,n+1} - \Delta t \vec{C}^{n+1}_{\vec{u}\vec{u}} \vec{J}_\mathcal{A}^\intercal(\vec{u}^{j,n+1}) \vecsym{\lambda}^{j,n+1} \\
=& \vec{u}^{j,n} + \Delta t \vec{C}^{n+1}_{\vec{u}\vec{w}} ( \Delta t \vec{C}^{n+1}_{\vec{w}\vec{w}} + \vecsym{\Gamma}^{-1})^{-1} (\vec{y}^{n+1} - \G(\vec{u}^{j,n})) \\
&- \Delta t \vec{C}^{n+1}_{\vec{u}\vec{u}} \vec{J}_\mathcal{A}^\intercal(\vec{u}^{j,n+1}) \vecsym{\lambda}^{j,n+1} + \mathcal{R}(\vec{u}^{j,n+1},\vecsym{\lambda}^{j,n+1})
\end{aligned}
\end{equation}
where $\mathcal{R}$ is a term of order $\mathcal{O}(\Delta t^2)$, in fact
$$
\mathcal{R}(\vec{u}^{j,n+1},\vecsym{\lambda}^{j,n+1}) = \Delta t^2 \vec{C}^{n+1}_{\vec{u}\vec{w}} ( \Delta t \vec{C}^{n+1}_{\vec{w}\vec{w}} + \vecsym{\Gamma}^{-1})^{-1} \vec{C}^{{n+1}^\intercal}_{\vec{u}\vec{w}} \vec{J}_\mathcal{A}^\intercal(\vec{u}^{j,n+1}) \vecsym{\lambda}^{j,n+1}.
$$
Now we interpret the parameter $\Delta t$ as an artificial time step for the iteration, i.e.~we take $\Delta t \sim N_t^{-1}$ where $N_t$ is the maximum number of iterations. Assume then $\vec{u}^{j,n} \approx \vec{u}^j(n\Delta t)$ and $\vecsym{\lambda}^{j,n} \approx \vecsym{\lambda}^j(n\Delta t)$ for $n\geq 0$ and $j=1,\dots,J$. Computing the limit $\Delta t\to 0^+$, \eqref{eq:updateControlScaled} is a first order implicit--explicit approximation of the following system of ordinary differential equations (ODEs):
\begin{equation} \label{eq:limitODE}
\frac{\mathrm{d}}{\mathrm{d} t} \vec{u}^j = \vec{C}_{\vec{u}\vec{w}} \vecsym{\Gamma} (\vec{y} - \G(\vec{u}^j)) - \vec{C}_{\vec{u}\vec{u}} \vec{J}_\mathcal{A}^\intercal(\vec{u}^j) \vecsym{\lambda}^j, \quad j=1,\dots,J.
\end{equation}
Imposing that the constraint needs to be satisfied at each time $t > 0$, we obtain the system of DAEs~\eqref{eq:limitDAE} endowed with initial conditions $\vec{u}^{j,0} = \vec{u}^j(0)\in\R^d$ such that $\mathcal{A}(\vec{u}^{j,0}) = \vec{0}_{\R^m}$.
\end{proof}

In the previous proposition we have seen~\eqref{eq:updateControl} as first order implicit--explicit time discretization of~\eqref{eq:limitDAE}. However, we stress again the fact that $\vec{C}^{n+1}_{\vec{u}\vec{u}}$ is actually computed with information at time level $n$, and therefore explicitly, while it appears in the term which is treated implicitly. Nevertheless, $\vec{C}^{n+1}_{\vec{u}\vec{u}}$ can be written in terms of the average quantities, which typically change in a small time scale, justifying therefore the use of an explicit evaluation.

\begin{corollary}
	Assume that $\G = \vec{G}$, with $\vec{G} \in \mathcal{L}(\R^d,\R^K)$, with $\mathcal{L}(\R^d,\R^K)$ space of linear operators mapping $\R^d$ to $\R^K$. The DAE system~\eqref{eq:limitDAE} can be written in terms of the gradient of the least squares functional $\Phi$~\eqref{eq:leastSqFnc}, obtaining
	\begin{equation} \label{eq:limitDAElinearG}
	\begin{aligned}
	\frac{\mathrm{d}}{\mathrm{d} t} \vec{u}^j &= -\vec{C}_{\vec{u}\vec{u}} (\nabla_{\vec{u}} \Phi (\vec{u}^j,\vec{y}) + \vec{J}_\mathcal{A}^\intercal(\vec{u}^j) \vecsym{\lambda}^j) \\
	\vec{0}_{\R^m} &= \mathcal{A}(\vec{u}^j(t)).
	\end{aligned}
	\end{equation}
\end{corollary} 
\begin{proof}
	Due to the linearity of the model $\vec{G}$, we observe that $\vec{C}_{\vec{u}\vec{w}} = \vec{C}_{\vec{u}\vec{u}} \vec{G}^\intercal$ and use $\nabla_{\vec{u}} \Phi (\vec{u}^j,\vec{y}) = -\vec{G}^\intercal \vecsym{\Gamma} (\vec{y} - \vec{G}\vec{u}^j)$ to rewrite~\eqref{eq:limitDAE} as~\eqref{eq:limitDAElinearG}.
\end{proof}

In view of the discussion in Section~\ref{sec:linearEqConstr}, it is easy to show that in the case of a linear equality constraint $\mathcal{A}(\vec{u}) = \vec{A} \vec{u} \in \R^m$, the DAE system~\eqref{eq:limitDAElinearG} reduces to
\begin{equation} \label{eq:unconstrGradientFlow}
	\frac{\mathrm{d}}{\mathrm{d} t} \vec{u}^j = - \vec{C}_{\vec{u}\vec{u}} \nabla_{\vec{u}} \Phi (\vec{u}^j,\vec{y})
\end{equation}
which is the preconditioned gradient flow equation for the least square functional $\Phi$, studied e.g.~in~\cite{SchillingsPreprint,schillingsstuart2017,schillingsstuart2018}. We expect that~\eqref{eq:unconstrGradientFlow} still allows each ensemble member to satisfy the linear constraint if the initial condition is feasible, again as consequence of the subspace property of the EnKF which holds also in the continuous dynamics~\cite{schillingsstuart2017}.

We observe that the differential equation in~\eqref{eq:limitDAElinearG} can be written as
\begin{equation} \label{eq:constrGradientFlow}
\begin{aligned}
\frac{\mathrm{d}}{\mathrm{d} t} \vec{u}^j &= - \vec{C}_{\vec{u}\vec{u}} \nabla_{\vec{u}} \Psi(\vec{u}^j,\vecsym{\lambda}^j,\vec{y}) \\
\Psi(\vec{u}^j,\vecsym{\lambda}^j,\vec{y}) &= \Phi (\vec{u}^j,\vec{y}) + \sum_{k=1}^m \vecsym{\lambda}^j_{(k)} \mathcal{A}_{(k)}(\vec{u}^j)
\end{aligned}
\end{equation}
and therefore it still has the structure of a preconditioned gradient type flow, where the flow $\Phi$ is perturbed along the direction of a linear combination of the constraints. However, while $\Phi$ is convex, $\Psi$ is not necessarily convex even if the constraint is since the method lacks information about the sign of the multipliers.

\subsubsection{Analysis in the case of a linear model} \label{sec:linearModel}

DAE systems are usually characterized by two indices, namely the perturbation index and the differentiation index. %The former indicates the influence of perturbations and their derivatives on the solution of the DAE, and therefore it concerns the stability. The latter indicates the number of differentiation required for the algebraic equation in order to write the DAE system as an ODE system. 
Special DAEs are the Hessenberg--type, having the property that perturbation and differentiation indices coincide. In particular, a DAE is said to be of high order index if their perturbation and differentiation indices are greater than or equal to $2$. Under suitable sufficient conditions it is possible to guarantee that a DAE has indices equal to $1$~\cite{GerdtsBook}.

\begin{proposition} \label{th:indexDAE}
	Let $I = [t_0,T] \subset \R$, $t_0 < T$ be a compact time interval. Let $\vec{x}:I\to\R^{d_\vec{x}}$ and $\vec{y}:I\to\R^{d_\vec{y}}$. Consider the semi--explicit DAE system
	\begin{align*}
	\frac{\mathrm{d}}{\mathrm{d}t} \vec{x}(t) = \vec{f}(t,\vec{x}(t),\vec{y}(t)),\\
	\vec{0}_{\R^{d_\vec{y}}} = \vec{g}(t,\vec{x}(t),\vec{y}(t)).
	\end{align*}
	Assume that
	\begin{itemize}
		\item[(a)] $\vec{f}$ is Lipschitz continuous with respect to $\vec{x}$ and $\vec{y}$ with Lipschitz constant $L_\vec{f}$ uniformly with respect to $t$;
		\item[(b)] $\vec{g}$ is continuously differentiable and $\vec{J}_\vec{g}^\vec{y} = [\nabla_\vec{y} g_1,\dots,\nabla_\vec{y} g_{d_\vec{y}}]^\intercal$ is non--singular and bounded for all $(t,\vec{x},\vec{y}) \in I\times \R^{d_\vec{x}} \times \R^{d_\vec{y}}$.
	\end{itemize}
	Then, the semi--explicit DAE system has perturbation index $1$. If additionally the inverse of $\vec{J}_\vec{g}^\vec{y}$ is bounded for all $(t,\vec{x},\vec{y}) \in I\times \R^{d_\vec{x}} \times \R^{d_\vec{y}}$, then the semi--explicit DAE system has differentiation index $1$.
\end{proposition}

%DAE systems with differentiation index $1$ are important since, in this case, one differentiation of the algebraic constraint would lead to a differential equation for the multipliers.
Using Proposition~\ref{th:indexDAE}, the following results hold true for~\eqref{eq:limitDAElinearG}, in which we recall that the dependence of the algebraic equation on the multipliers is given implicitly by the differential variables, namely $\mathcal{A}(\vec{u}_j(t))$ is in fact $\mathcal{A}(\vec{u}_j(t;\vecsym{\lambda}_j))$.

\begin{corollary} \label{th:indexGradientFlow}
	Consider the semi--explicit DAE system~\eqref{eq:limitDAElinearG} for each fixed $j=1,\dots,J$. Assume that condition (b) in Proposition~\ref{th:indexDAE} holds for the vector valued function $\mathcal{A}$ with respect to the algebraic variables $\vecsym{\lambda}$. Then, system~\eqref{eq:limitDAElinearG} has perturbation index $1$.  If additionally the inverse of $\vec{J}_\mathcal{A}^{\vecsym{\lambda}^j}$ is bounded for all $(\vec{u}^j,\vecsym{\lambda}^j) \in \R^{d} \times \R^{m}$, then system~\eqref{eq:limitDAElinearG} has differentiation index $1$.
\end{corollary}
\begin{proof}
	Continuous differentiability of the vector valued function $\mathcal{A}$ with respect to the dynamical variable $\vec{u}^j$ gives a sufficient condition to the right hand side of the dynamical equation in~\eqref{eq:limitDAElinearG} to be Lipschitz continuous with respect to $\vec{u}^j$ and $\vecsym{\lambda}^j$, for each fixed $j=1,\dots,J$. Then, the statement follows as application of Proposition~\ref{th:indexDAE}.
\end{proof}

\begin{corollary}
	For each fixed $j=1,\dots,J$, consider the DAE system~\eqref{eq:limitDAElinearG} with initial condition $\vec{u}^j(0) = \vec{u}^{j,0}$. Let $(\tilde{\vec{u}}^j,\tilde{\vecsym{\lambda}}^j)$ be the solution of the perturbed system
	\begin{align*}
		\frac{\mathrm{d}}{\mathrm{d}t} \tilde{\vec{u}}^j &= -\vec{C}_{\tilde{\vec{u}}\tilde{\vec{u}}} \nabla_{\tilde{\vec{u}}} \Psi(\tilde{\vec{u}}^j,\tilde{\vecsym{\lambda}}^j,\vec{y}) + \vecsym{\delta}_1(t), \quad \tilde{\vec{u}}^j(0) = \tilde{\vec{u}}^{j,0} \\
		\vec{0}_{\R^m} &= \mathcal{A}(\tilde{\vec{u}}^j) + \vecsym{\delta}_2(t) 
	\end{align*}
	on $t\in[0,T]\subset\R$ and with $\Psi$ defined in~\eqref{eq:constrGradientFlow}. Then, under the assumptions of Corollary~\ref{th:indexGradientFlow}, $\exists\,L_1,L_2\geq 0$ such that the following bound holds true:
	$$
		\left\| \vec{u}^j(t) - \tilde{\vec{u}}^j(t) \right\| \leq \left( \left\| \vec{u}^{j,0} - \tilde{\vec{u}}^{j,0} \right\| + T\left( L_1 L_2 \max_{0\leq\tau\leq t} \left\| \vecsym{\delta}_2(\tau) \right\| + \max_{0\leq\tau\leq t} \left\| \vecsym{\delta}_1(\tau) \right\| \right) \right) \exp\left( L_1(1+L_2)t \right)
	$$
\end{corollary}
\begin{proof}
	The assumptions of Corollary~\ref{th:indexGradientFlow} guarantee that the implicit function theorem can be applied to the algebraic equation $\vec{0}_{\R^m}=\mathcal{A}(\tilde{\vec{u}}^j)+\vecsym{\delta}_2$. Therefore, we can solve for $\vecsym{\lambda}^j\in\R^m$, $\forall\,t\in[0,T]$, $\tilde{\vec{u}}^j\in\R^d$, obtaining $\tilde{\vecsym{\lambda}}^j = \vecsym{\Lambda}(\tilde{\vec{u}}^j,\vecsym{\delta}_2)$. Moreover, $\vecsym{\Lambda}$ is locally Lipschitz continuous with respect to $\tilde{\vec{u}}^j$ and $\vecsym{\delta}_2$ with Lipschitz constant $L_2$. Then for the multipliers we get the bound
	$$
		\left\| \vecsym{\lambda}^j(t) - \tilde{\vecsym{\lambda}}^j(t) \right\| = \left\| \vecsym{\Lambda}(\vec{u}^j(t),\vec{0}_{\R^m}) - \vecsym{\Lambda}(\tilde{\vec{u}}^j(t),\vecsym{\delta}_2) \right\| \leq L_2 \left( \left\| \vec{u}^j(t) - \tilde{\vec{u}}^j(t) \right\| + \left\| \vecsym{\delta}_2(t) \right\| \right).
	$$
	Let $L_1$ be the Lipschitz constant of the right--hand side of the dynamical equation. We have
	\begin{align*}
		\left\| \vec{u}^j(t) - \tilde{\vec{u}}^j(t) \right\| \leq& \left\| \vec{u}^{j,0} - \tilde{\vec{u}}^{j,0} \right\| + L_1 \int_0^t \left\| \vec{u}^j(\tau) - \tilde{\vec{u}}^j(\tau) \right\| + \left\| \vecsym{\lambda}^j(\tau) - \tilde{\vecsym{\lambda}}^j(\tau) \right\| \mathrm{d}\tau + \left\| \int_0^t \vecsym{\delta}_1(\tau) \mathrm{d}\tau \right\| \\
		\leq& \left\| \vec{u}^{j,0} - \tilde{\vec{u}}^{j,0} \right\| + L_1 \left( 1+L_2 \right) \int_0^t \left\| \vec{u}^j(\tau) - \tilde{\vec{u}}^j(\tau) \right\| \mathrm{d}\tau \\
		&+ L_1 L_2 \int_0^t \left\| \vecsym{\delta}_2(\tau) \right\| \mathrm{d}\tau + \int_0^t \left\| \vecsym{\delta}_1(\tau) \right\| \mathrm{d}\tau.
	\end{align*}
	Then the statement follows easily by application of the Gronwall's lemma.
\end{proof}

The DAE system~\eqref{eq:limitDAElinearG} is derived by starting from the first order necessary optimality conditions stated in Proposition~\ref{th:optimality}. We recall that these conditions are also sufficient if the feasible set is convex, see Corollary~\ref{th:convex}. Now we study the large time behavior of the solution of the DAE system~\eqref{eq:limitDAElinearG}.

Let us introduce the following notation. We define $\vec{\overline{m}}(t)$ and $\doverline{\vec{m}}(t)$ be the mean and the second moment of the ensembles, respectively, at time $t\geq 0$, namely
$$
\vec{\overline{m}}(t) = \frac{1}{J} \sum_{j=1}^J \vec{u}^j(t), \quad \doverline{\vec{m}}(t) = \frac{1}{J} \sum_{j=1}^J \vec{u}^j(t) \otimes \vec{u}^j(t).
$$
Further, for each $j=1,\dots,J$, we define
\begin{equation} \label{eq:spreadresidual}
\vec{e}^j(t) = \vec{u}^j(t) - \vec{\overline{m}}(t), \quad \vec{r}^j(t) = \vec{u}^j(t) - \vec{u}^* %\quad \vec{E}^j(t) = \vecsym{\Gamma}^{\frac12} \vec{G} \vec{e}^j(t),
\end{equation}
be the ensemble spread %, the projected ensemble spread into the observation space
and the residual to a value $\vec{u}^*$, respectively.
We observe that the evolution in time of $\vec{\overline{m}}$ and $\doverline{\vec{m}}$ are governed by
\begin{equation} \label{eq:ensembleMean}
\begin{aligned}
\frac{\mathrm{d}}{\mathrm{d} t} \vec{\overline{m}}(t) &= \frac{1}{J} \sum_{j=1}^J \frac{\mathrm{d}}{\mathrm{d} t} \vec{u}^j(t) = - \frac{1}{J} \sum_{j=1}^J \vec{C}_{\vec{u}\vec{u}} \left( \nabla_{\vec{u}} \Phi(\vec{u}^j,\vec{y}) + \vec{J}_\mathcal{A}^\intercal(\vec{u}^j) \vecsym{\lambda}^j \right)\\
&= -\vec{C}_{\vec{u}\vec{u}} \left( \nabla_{\vec{u}} \Phi\left(\frac{1}{J}\sum_{j=1}^J \vec{u}^j,\vec{y}\right)
+ \frac{1}{J} \sum_{j=1}^J \vec{J}_\mathcal{A}^\intercal(\vec{u}^j) \vecsym{\lambda}^j \right)\\
%&= -\vec{C}_{\vec{u}\vec{u}} \left( \nabla_{\vec{u}} \Phi(\vec{\overline{m}},\vec{y}) + \frac{1}{J} \sum_{j=1}^J \vec{J}_\mathcal{A}^\intercal(\vec{u}^j) \vecsym{\lambda}^j \right) \\
&= -\vec{C}_{\vec{u}\vec{u}} \left( \nabla_{\vec{u}} \Phi(\vec{\overline{m}},\vec{y}) + \sum_{k=1}^m \frac{1}{J} \sum_{j=1}^J \vecsym{\lambda}^j_{(k)} \nabla_{\vec{u}} \mathcal{A}_{(k)}(\vec{u}^j) \right), \\
\frac{\mathrm{d}}{\mathrm{d} t} \doverline{\vec{m}}(t) &= \frac{1}{J} \sum_{j=1}^J \vec{u}^j(t) \otimes \vec{u}^j(t) = \frac{2}{J} \sum_{j=1}^J \left(\frac{\mathrm{d}}{\mathrm{d}t} \vec{u}^j(t) \right) \otimes \vec{u}^j(t) \\
&= - \frac{2}{J} \sum_{j=1}^J \vec{C}_{\vec{u}\vec{u}} \left( \nabla_{\vec{u}} \Phi(\vec{u}^j,\vec{y}) + \vec{J}_\mathcal{A}^\intercal(\vec{u}^j) \vecsym{\lambda}^j \right) \otimes \vec{u}^j(t)\\
&= -2 \vec{C}_{\vec{u}\vec{u}} \left( \nabla_{\vec{u}} \Phi\left(\frac{1}{J}\sum_{j=1}^J \vec{u}^j \otimes \vec{u}^j,\vec{y} \otimes \frac{1}{J}\sum_{j=1}^J \vec{u}^j \right)
+ \frac{1}{J} \sum_{j=1}^J \vec{J}_\mathcal{A}^\intercal(\vec{u}^j) \vecsym{\lambda}^j \otimes \vec{u}^j \right) \\
&= - 2 \vec{C}_{\vec{u}\vec{u}} \left( \nabla_{\vec{u}} \Phi\left(\doverline{\vec{m}},\vec{y}\otimes\vec{\overline{m}}\right) + \frac{1}{J} \sum_{j=1}^J \vec{J}_\mathcal{A}^\intercal(\vec{u}^j) \vecsym{\lambda}^j \otimes \vec{u}^j \right)
\end{aligned}
\end{equation}
where we have used the linearity of $\nabla_{\vec{u}}\Phi$. While the evolution in time of $\vec{e}^j$ is governed by
\begin{equation} \label{eq:ensembleSpread}
\begin{aligned}
\frac{\mathrm{d}}{\mathrm{d} t} \vec{e}^j(t) &= -\vec{C}_{\vec{u}\vec{u}} \left( \vec{G}^\intercal \vecsym{\Gamma} \vec{G} \vec{e}^j + \vec{J}_\mathcal{A}^\intercal(\vec{u}^j) \vecsym{\lambda}^j - \frac{1}{J} \sum_{\ell=1}^J \vec{J}_\mathcal{A}^\intercal(\vec{u}^\ell) \vecsym{\lambda}^\ell \right) \\
&= -\vec{C}_{\vec{u}\vec{u}} \left( \vec{G}^\intercal \vecsym{\Gamma} \vec{G} \vec{e}^j + \sum_{k=1}^m \vecsym{\lambda}^j_{(k)} \nabla_{\vec{u}} \mathcal{A}_{(k)}(\vec{u}^j) - \sum_{k=1}^m \frac{1}{J} \sum_{\ell=1}^J \vecsym{\lambda}^\ell_{(k)} \nabla_{\vec{u}} \mathcal{A}_{(k)}(\vec{u}^\ell) \right), \quad j=1,\dots,J,
\end{aligned}
\end{equation}
respectively. The covariance matrix can be written in terms of the mean and the second moment of the ensembles as
\begin{equation} \label{eq:covariance}
	\vec{C}_{\vec{u}\vec{u}}(t) : \left( \vec{u}^1(t), \dots, \vec{u}^J(t) \right) \in \R^d\times\dots\times\R^d \longmapsto \doverline{\vec{m}}(t) - \vec{\overline{m}}(t) \otimes \vec{\overline{m}}(t)\in\R^{d\times d}.
\end{equation}
Equations~\eqref{eq:ensembleMean} and~\eqref{eq:ensembleSpread} are coupled with $J$ algebraic equations imposing the validity of the constraint for each ensemble member. In particular, we stress the fact that the systems for the ensemble mean, second moment and spread are not closed, due to the presence of the multipliers and since the constraint is nonlinear and needs to be satisfied by each ensemble member. This requires the knowledge of the evolution in time of the ensembles. Further, it is not ensured that the mean of the ensemble satisfies the constraint. 

We prove the following results for the residuals in the space of the control.

\begin{proposition} \label{th:ensembleKKT}
%	Let $\vec{u}^*$ be an optimal solution of the minimization problem $$\min_{\vec{u}\in\R^d} \Phi(\vec{u},\vec{y}) \ \mbox{ subject to } \ \mathcal{A}(\vec{u}) = \vec{0}_{\R^m}$$ for a given $\vec{y}\in\R^K$. Let $\vec{u}^{j,0} \in \R^d$ be an initial condition of the DAE system~\eqref{eq:limitDAElinearG} such that $\mathcal{A}(\vec{u}^{j,0}) = \vec{0}_{\R^m}$, for all $j=1,\dots,J$. Then the steady state $(\vec{u}^{j,\infty},\vecsym{\lambda}^{j,\infty}) \in \R^d\times\R^m$ is a KKT point of the minimization problem provided that $\vec{C}_{\vec{u}\vec{u}}^\infty$ is positive definite. Moreover, if $\mathcal{A}(\vec{u}) = \vec{0}_{\R^m}$ is a convex feasible domain, then $\vec{u}^{j,\infty}$ provides a first order approximation of $\vec{u}^*$ and $\| \vec{r}^j(t) \| \to 0$ as $t \to \infty$.
	Let $\vec{u}^{j,0} \in \R^d$ be an initial condition of the DAE system~\eqref{eq:limitDAElinearG} such that $\mathcal{A}(\vec{u}^{j,0}) = \vec{0}_{\R^m}$, for all $j=1,\dots,J$. Then there exists a steady state $(\vec{u}^{j,\infty},\vecsym{\lambda}^{j,\infty}) \in \R^d\times\R^m$ of the DAE dynamics which is a KKT point of the minimization problem
	$$\min_{\vec{u}\in\R^d} \Phi(\vec{u},\vec{y}) \ \mbox{ subject to } \ \mathcal{A}(\vec{u}) = \vec{0}_{\R^m}$$
	for a given $\vec{y}\in\R^K$.
	Moreover, if $\mathcal{A}(\vec{u}) = \vec{0}_{\R^m}$ is a convex feasible domain, then $\vec{u}^{j,\infty}$ provides a first order approximation of $\vec{u}^*$ being an optimal solution of the minimization problem.
\end{proposition}
\begin{proof}
	We have that a steady state $(\vec{u}^{j,\infty},\vecsym{\lambda}^{j,\infty})$ of the DAE system~\eqref{eq:limitDAElinearG} solves
	$$
	\vec{0}_{\R^d} = \nabla_{\vec{u}} \Phi (\vec{u}^{j,\infty},\vec{y}) + \sum_{k=1}^m \vecsym{\lambda}^{j,\infty}_{(k)} \nabla_{\vec{u}} \mathcal{A}_{(k)}(\vec{u}^{j,\infty}), \quad \vec{0}_{\R^m} = \mathcal{A}(\vec{u}^{j,\infty}).
	$$
	For a given $\vec{y}\in\R^K$, this is a KKT system and they are the first order necessary optimality conditions for $\vec{u}^{j,\infty}$ as solution to the minimization problem $\min_{\vec{u}\in\R^d} \Phi(\vec{u},\vec{y})$ subject to $\mathcal{A}(\vec{u}) = \vec{0}_{\R^m}$. If the set $\mathcal{A}(\vec{u}) = \vec{0}_{\R^m}$ is convex, the KKT conditions are also sufficient. Then $\vec{u}^{j,\infty}$ is a solution of the minimization problem.
\end{proof}

Since the typical estimator of the EnKF is provided by the mean of the ensemble, we discuss properties of the ensemble mean.

\begin{proposition} \label{th:meanKKT}
%	Let $\vec{u}^* = \{ \vec{u}^{j,*} \}_{j=1}^J$ be an optimal solution of the minimization problem $$\min_{\vec{u}^1,\dots,\vec{u}^J} \Phi(\vec{\overline{m}},\vec{y}) \ \mbox{ subject to } \ \mathcal{A}(\vec{u}^1) = \dots = \mathcal{A}(\vec{u}^J) = \vec{0}_{\R^m}$$ for a given $\vec{y}\in\R^K$. Let $\vec{u}^0 = \{ \vec{u}^{j,0} \}_{j=1}^J$ be an initial condition such that $\mathcal{A}(\vec{u}^{j,0}) = \vec{0}_{\R^m}$, for all $j=1,\dots,J$. Then the steady state $(\vec{\overline{m}}^{\infty},\vecsym{\lambda}^{j,\infty}) \in \R^d\times\R^m$ is a KKT point of the minimization problem provided that $\vec{C}_{\vec{u}\vec{u}}^\infty$ is positive definite. Moreover, if $\mathcal{A}(\vec{u}) = \vec{0}_{\R^m}$ is a convex feasible domain, then $\vec{\overline{m}}^{\infty}$ provides a first order approximation of $\vec{u}^*$ and $\| \vec{\overline{m}}(t) - \vec{u}^* \| \to 0$ as $t \to \infty$.
	%
	Let $\vec{u}^0 = \{ \vec{u}^{j,0} \}_{j=1}^J$ be an initial condition of the DAE system~\eqref{eq:limitDAElinearG} such that $\mathcal{A}(\vec{u}^{j,0}) = \vec{0}_{\R^m}$, for all $j=1,\dots,J$. Then there exists a steady state $(\vec{\overline{m}}^{\infty},\vecsym{\lambda}^{1,\infty},\dots,\vecsym{\lambda}^{J,\infty}) \in \R^d\times\R^m\times\cdots\times\R^m$ of the dynamical system~\eqref{eq:ensembleMean} of the first moment $\vec{\overline{m}}$ which is a KKT point of the minimization problem
	$$\min_{\vec{u}^1,\dots,\vec{u}^J} \Phi(\vec{\overline{m}},\vec{y}) \ \mbox{ subject to } \ \mathcal{A}(\vec{u}^1) = \dots = \mathcal{A}(\vec{u}^J) = \vec{0}_{\R^m}$$
	for a given $\vec{y}\in\R^K$.
	Moreover, if $\mathcal{A}(\vec{u}) = \vec{0}_{\R^m}$ is a convex feasible domain, then $\vec{\overline{m}}^{\infty}$ provides a first order approximation of an optimal solution of the minimization problem.
\end{proposition}
\begin{proof}
	We have that a steady state $(\vec{\overline{m}}^{\infty},\vecsym{\lambda}^{1,\infty},\dots,\vecsym{\lambda}^{J,\infty})$ of~\eqref{eq:ensembleMean} solves
	$$
	\vec{0}_{\R^d} = \nabla_{\vec{u}} \Phi (\vec{\overline{m}}^{\infty},\vec{y}) + \sum_{k=1}^m \frac{1}{J} \sum_{j=1}^J \vecsym{\lambda}^{j,\infty}_{(k)} \nabla_{\vec{u}} \mathcal{A}_{(k)}(\vec{u}^{j,\infty}), \quad \vec{0}_{\R^m} = \mathcal{A}(\vec{u}^{J,\infty}) = \dots = \mathcal{A}(\vec{u}^{1,\infty})
	$$
	where the $\vec{u}^{j,\infty}$'s are the steady states of~\eqref{eq:limitDAElinearG}. For a given $\vec{y}\in\R^K$, this is a KKT system and they are the first order necessary optimality conditions for $\vec{\overline{m}}^{\infty}$ as solution to the minimization problem $\min_{\vec{u}} \Phi(\vec{\overline{m}},\vec{y})$ subject to $\mathcal{A}(\vec{u}^1) = \dots = \mathcal{A}(\vec{u}^J) = \vec{0}$. If the set $\mathcal{A}(\vec{u}) = \vec{0}_{\R^m}$ is convex, the KKT conditions are also sufficient. Then $\vec{\overline{m}}^{\infty}$ is a solution of the minimization problem.
\end{proof}

\begin{remark}
	The optimization problems in Proposition~\ref{th:ensembleKKT} and Proposition~\ref{th:meanKKT} are the same if the constraint is linear.
\end{remark}

From Proposition~\ref{th:meanKKT} it is clear that the steady state of the mean of the ensembles depends on the steady states of the ensembles. Moreover, Proposition~\ref{th:meanKKT} does not guarantee that the mean of the ensembles  at equilibrium satisfies the constraint. A sufficient condition to guarantee that $\mathcal{A}(\vec{\overline{m}}^\infty)=\vec{0}_{\R^m}$ is that the KKT point belongs to the set $\vec{C}_{\vec{u}\vec{u}}^\infty = \vec{0}_{\R^{d\times d}}$, which is a set of possible equilibria for~\eqref{eq:limitDAElinearG} and~\eqref{eq:ensembleMean}. In fact, we note that a concentration of the particles at any point of $\R^d$ is a stationary solution of the dynamics. In this case $\left\| \vec{e}^{j,\infty} \right\| = 0$, for all $j=1,\dots,J$. Therefore the question whether or not all the equilibria are necessary in the kernel of $\vec{C}_{\vec{u}\vec{u}}^\infty$ is to be discussed. This is true in the unconstrained ensemble Kalman filter as proved in~\cite{HertyVisconti2019}.

The following counterexample shows that not all the equilibrium solutions of~\eqref{eq:limitDAElinearG} belong to the set $\vec{C}_{\vec{u}\vec{u}} = \vec{0}_{\R^{d\times d}}$. We consider the case of a one--dimensional control and two ensembles $u, v\in\R$, so that $d=1$ and $J=2$. In addition, we take a scalar quadratic and convex constraint $\mathcal{A}(u) = \frac{1}{2} h_1 u^2 + h_2 u$, with $h_1>0$. In order to study the steady state of the mean and the second moment of the ensembles, we need to couple the evolution equations~\eqref{eq:ensembleMean} with the system of DAEs~\eqref{eq:limitDAElinearG}, obtaining
\begin{equation} \label{eq:oneDimCase}
	\begin{aligned}
		\frac{\mathrm{d}}{\mathrm{d} t} u &= - C_{uv} \left( -G^\intercal \Gamma (y-G u) + h_1 \lambda u + h_2 \lambda \right), \\
		\frac{\mathrm{d}}{\mathrm{d} t} v &= - C_{uv} \left( -G^\intercal \Gamma (y-G v) + h_1 \mu v + h_2 \mu \right), \\
		0 &= \mathcal{A}(u) = \mathcal{A}(v), \\
		\frac{\mathrm{d}}{\mathrm{d} t} \overline{m} &= - C_{uv} \left( -G^\intercal \Gamma (y-G \overline{m}) + \frac{h_1}{2} \left( \lambda u + \mu v \right) + h_2 \left( \lambda + \mu \right) \right), \\
		\frac{\mathrm{d}}{\mathrm{d} t} \doverline{m} &= - 2 C_{uv} \left( -G^\intercal \Gamma (y\overline{m}-G \doverline{m}) + \frac{h_1}{2} \left( \lambda u^2 + \mu v^2 \right) + h_2 \left( \lambda u + \mu v \right) \right).
	\end{aligned}
\end{equation}
Here, $\lambda$ and $\mu$ are the two multipliers related to $u$ and $v$, respectively. We compute the steady states of~\eqref{eq:oneDimCase} by solving $\dot{u}=0$, $\dot{v}=0$, $\mathcal{A}(u)=0$, $\mathcal{A}(v)=0$, $\dot{\overline{m}}=0$, $\dot{\doverline{m}}=0$ obtaining the nullclines in the phase space $(u,v,\lambda,\mu,\overline{m},\doverline{m})$. Solutions are provided by either $C_{uv} = \doverline{m} - \overline{m}^2 = 0$ or
\begin{equation} \label{eq:solutionsEq}
	\begin{aligned}
		u &= \left( G^\intercal \Gamma G + h_1 \lambda \right)^{-1} \left( G^\intercal \Gamma y - h_2 \lambda \right), \\
		v &= \left( G^\intercal \Gamma G + h_1 \mu \right)^{-1} \left( G^\intercal \Gamma y - h_2 \mu \right), \\
		\overline{m} &= \left( G^\intercal \Gamma G \right)^{-1} \left( G^\intercal \Gamma y - \frac{h_1}{2} \left( \lambda u + \mu v \right) - \frac{h_2}{2} \left( \lambda + \mu \right) \right), \\
		\doverline{m} &= \left( G^\intercal \Gamma G \right)^{-1} \left( G^\intercal \Gamma y \overline{m} - \frac{h_1}{2} \left( \lambda u^2 + \mu v^2 \right) - \frac{h_2}{2} \left( \lambda u + \mu v \right) \right),
	\end{aligned}
\end{equation}
where $\lambda$ and $\mu$ are solutions $\mathcal{A}(u)=0$ and $\mathcal{A}(v)=0$, respectively. The equilibrium or fixed points are the intersections of the nullclines and therefore the question becomes whether or not $(\overline{m},\doverline{m})$ always belongs to the set where $\doverline{m}-\overline{m}^2=0$. It is easy to observe that this is possible if and only if $\lambda = \mu$, since then
\begin{align*}
	\overline{m} &= \left( G^\intercal \Gamma G + h_1 \lambda \right)^{-1} \left( G^\intercal \Gamma y - h_2 \lambda \right), \\
	\doverline{m} &= \left( G^\intercal \Gamma G + h_1 \lambda \right)^{-1} \left( G^\intercal \Gamma y - h_2 \lambda \right) \overline{m}.
\end{align*}
We can then conclude that this counterexample shows that, even in the simplest one--dimensional setting, the constrained EnKF provides feasible solutions to the constrained optimization problem which do not collapse.
This consideration opens the question on guaranteeing that the collapse of the ensemble to the mean occurs, i.e.~the equilibria lie on the set where $\left\| \vec{e}^{j,\infty} \right\| = 0$, for $j=1,\dots,J$.
%For instance, in the case of quadratic constraints, one might reformulate the DAE system~\eqref{eq:limitDAElinearG} as a singularly perturbed ODE system
%\begin{equation} \label{eq:singPertODE}
%\begin{aligned}
%	\frac{\mathrm{d}}{\mathrm{d} t} \vec{u}^j &= -\vec{C}_{\vec{u}\vec{u}} (\nabla_{\vec{u}} \Phi (\vec{u}^j,\vec{y}) + \vec{J}_\mathcal{A}^\intercal(\vec{u}^j) \vecsym{\lambda}^j) \\
%	\varepsilon \frac{\mathrm{d}}{\mathrm{d} t} \vecsym{\lambda}^j &= \mathcal{A}(\vec{u}^j(t))
%\end{aligned}
%\end{equation}
%where $\varepsilon$ is a positive and small parameter. When $\varepsilon = 0$, system~\eqref{eq:singPertODE} formally converges to~\eqref{eq:limitDAElinearG}. The regularization of the algebraic equation into a differential equation for the multipliers guarantees the convergence to the stable steady state. However, for $\varepsilon$ very small this system is stiff and its numerical solution requires methods for stiff ODEs.

A sufficient condition for the existence of a monotonic decay to a bound for the ensemble spread is guaranteed by the following result, which holds in the space of the control.

\begin{proposition} \label{th:spread}
	Let $\vec{u}^{j,0}\in\R^d$, $j=1,\dots,J$, be an initial condition of the DAE system~\eqref{eq:limitDAElinearG} such that $\mathcal{A}(\vec{u}^{j,0})=\vec{0}_{\R^m}$. Assume that $\mathcal{A} : \R^d \to \R^m$ is convex and quadratic. Then the quantity $\frac{1}{J} \sum_{j=1}^J \left\| \vec{e}^j(t) \right\|^2$ is decreasing in time and thus in particular we have
	$
	\frac{1}{J} \sum_{j=1}^J \left\| \vec{e}^j(t) \right\|^2 \leq \frac{1}{J} \sum_{j=1}^J \left\| \vec{e}^j(0) \right\|^2$, for $t \geq 0$, provided that $\vecsym{\lambda}^j(t) = \vecsym{\Lambda}(t) \geq 0$, $\forall\,j=1,\dots,J$ and $t>0$.
\end{proposition}
\begin{proof}
	For the sake of simplicity we consider a scalar constraint, i.e.~$m=1$. To prove the statement, it is sufficient to show that $\frac12 \frac{\mathrm{d}}{\mathrm{d}t} \frac{1}{J} \sum_{j=1}^J \left\| \vec{e}^j(t) \right\|^2 \leq 0$. Using~\eqref{eq:ensembleSpread}, we compute
	\begin{align*}
	\frac12 \frac{\mathrm{d}}{\mathrm{d}t} \frac{1}{J} \sum_{j=1}^J \left\| \vec{e}^j(t) \right\|^2 =& \frac{1}{J} \sum_{j=1}^J \left\langle \vec{e}^j(t) , \frac{\mathrm{d}}{\mathrm{d}t} \vec{e}^j(t) \right\rangle \\
	=& - \frac{1}{J} \sum_{j=1}^J \left\langle \vec{e}^j(t) , \vec{C}_{\vec{u}\vec{u}} \vec{G}^\intercal \vecsym{\Gamma} \vec{G} \vec{e}^j(t) \right\rangle \\
	&- \frac{1}{J} \sum_{j=1}^J \left\langle \vec{e}^j(t) , \vec{C}_{\vec{u}\vec{u}} \left( \lambda^j \nabla_{\vec{u}} \mathcal{A}(\vec{u}^j) - \frac{1}{J} \sum_{\ell=1}^J \lambda^\ell \nabla_{\vec{u}} \mathcal{A}(\vec{u}^\ell) \right) \right\rangle.
	\end{align*}
	Let us consider the first term in the right hand side. Using the structure of $\vec{C}_{\vec{u}\vec{u}}$ we have
	$$
	- \frac{1}{J} \sum_{j=1}^J \left\langle \vec{e}^j(t) , \vec{C}_{\vec{u}\vec{u}} \vec{G}^\intercal \vecsym{\Gamma} \vec{G} \vec{e}^j(t) \right\rangle = - \frac{1}{J^2} \sum_{k,j=1}^J \left\langle \vec{e}^j(t) , \vec{e}^k(t) \right\rangle \left\langle \vec{e}^k(t) , \vec{G}^\intercal \vecsym{\Gamma} \vec{G} \vec{e}^j(t) \right\rangle.
	$$
	Since $\vec{G}^\intercal \vecsym{\Gamma} \vec{G}$ is symmetric and positive semidefinite, it is possible to find a set of eigenpairs $(\mu^i,\vec{v}^i)_{i=1}^d$ with $\mu^i \geq 0$ and $\vec{v}^i$ orthonormal basis such that $\vec{e}^j = \sum_{i=1}^d \alpha^{i_j} \vec{v}^i$, $\forall\,j=1,\dots,J$, and $\vec{G}^\intercal \vecsym{\Gamma} \vec{G} \vec{v}^i = \mu^i \vec{v}^i$, for all $i=1,\dots,d$. Then
	$$
	- \frac{1}{J^2} \sum_{k,j=1}^J \left\langle \vec{e}^j(t) , \vec{e}^k(t) \right\rangle \left\langle \vec{e}^k(t) , \vec{G}^\intercal \vecsym{\Gamma} \vec{G} \vec{e}^j(t) \right\rangle = - \frac{1}{J^2} \sum_{k,j=1}^J \left( \sum_{i=1}^d \sqrt{\mu^i} \alpha^{i_k} \alpha^{i_j} \right)^2 \leq 0.
	$$
	For the second term in the right hand side, using the assumption $\lambda^j(t) = \Lambda(t) \geq 0$ and the convexity of the constraint, we obtain
	$$
	- \frac{1}{J} \sum_{j=1}^J \left\langle \vec{e}^j(t) , \vec{C}_{\vec{u}\vec{u}} \left( \lambda^j \nabla_{\vec{u}} \mathcal{A}(\vec{u}^j) - \frac{1}{J} \sum_{\ell=1}^J \lambda^\ell \nabla_{\vec{u}} \mathcal{A}(\vec{u}^\ell) \right) \right\rangle = - \frac{\Lambda}{J} \sum_{j=1}^J \left\langle \vec{e}^j(t) , \vec{C}_{\vec{u}\vec{u}} \nabla_{\vec{u}} \mathcal{A}(\vec{e}^j) \right\rangle \leq 0.
	$$\qedhere
\end{proof}

Under stronger assumptions, it is possible to prove convergence of the residual in the control space.

\begin{proposition}
	Let $\vec{u}^{j,0}\in\R^d$, $j=1,\dots,J$, be an initial condition of the DAE system~(18) such that $\mathcal{A}(\vec{u}^{j,0})=\vec{0}_{\R^m}$. Assume that $\mathcal{A}(\vec{u}) = \vec{0}_{\R^m}$ is a convex feasible domain and $\mathcal{A} : \R^d \to \R^m$ is convex and quadratic. Assume that $\vec{G}^\intercal \vecsym{\Gamma} \vec{G}$ is positive definite. Then $\lim_{t\to\infty} \frac{1}{J} \sum_{j=1}^J \left\| \vec{r}^j(t) \right\|^2 = 0$, provided that $\vecsym{\lambda}^j(t) = \vecsym{\Lambda}(t) \geq 0$, $\forall\,j=1,\dots,J$ and $t>0$.
\end{proposition}
\begin{proof}
	By assumption $\vec{G}^\intercal \vecsym{\Gamma} \vec{G}$ is positive definite and thus we have a unique global minimizer $\vec{u}^*$ of the constrained minimization problem
	$$\min_{\vec{u}\in\R^d} \Phi(\vec{u},\vec{y}) \ \mbox{ subject to } \ \mathcal{A}(\vec{u}) = \vec{0}_{\R^m}$$
	for a given $\vec{y}\in\R^K$.
	For simplicity and without loss of generality we consider the case of a scalar constraint, i.e.~$m=1$. We study
	\begin{align*}
	\frac{\mathrm{d}}{\mathrm{d}t} \frac{1}{2 J} \sum_{j=1}^J \left\| \vec{r}^j(t) \right\|^2 =& \frac{\mathrm{d}}{\mathrm{d}t} \frac{1}{2 J} \sum_{j=1}^J \left\| \vec{u}^j(t) - \vec{u}^* \right\|^2 = \frac{1}{J} \sum_{j=1}^J \left\langle \vec{u}^j(t) - \vec{u}^* , \frac{\mathrm{d}}{\mathrm{d}t} \vec{u}^j(t) \right\rangle \\
	=& - \frac{1}{J} \sum_{j=1}^J \left\langle \vec{u}^j(t) - \vec{u}^* , \vec{C}_{\vec{u}\vec{u}} \left( \nabla_\vec{u}(\vec{u}^j,\vec{y}) + \lambda^j \nabla_{\vec{u}} \mathcal{A}(\vec{u})|_{\vec{u}=\vec{u}^j(t)} \right) \right\rangle \\
	%=& \frac{1}{J} \sum_{j=1}^J \left\langle \vec{u}^j - \vec{u}^* , \vec{C}_{\vec{u}\vec{u}} \vec{G}^\intercal \vecsym{\Gamma} \vec{G} (\vec{y}-\vec{u}^j) \right\rangle - \frac{1}{J} \sum_{j=1}^J \left\langle \vec{u}^j - \vec{u}^* , \lambda^j \vec{C}_{\vec{u}\vec{u}} \nabla_{\vec{u}}^\intercal \mathcal{A}(\vec{u})|_{\vec{u}=\vec{u}^j} \right\rangle \\
	=& - \frac{1}{J} \sum_{j=1}^J \left\langle \vec{u}^j - \vec{u}^* , \vec{C}_{\vec{u}\vec{u}} \vec{G}^\intercal \vecsym{\Gamma} \vec{G} (\vec{u}^j-\vec{u}^*) \right\rangle \\
	&- \frac{\Lambda}{J} \sum_{j=1}^J \left\langle \vec{u}^j - \vec{u}^* , \lambda^j \vec{C}_{\vec{u}\vec{u}} \nabla_{\vec{u}}\left(\mathcal{A}(\vec{u})|_{\vec{u}=\vec{u}^j} - \mathcal{A}(\vec{u}^*) \right) \right\rangle \\
	&< 0.
	%=& - \frac{1}{J} \sum_{j=1}^J (\vec{u}^j - \vec{u}^*)^\intercal (\vec{u}^j - \overline{\vec{m}}) \vec{G}^\intercal \vecsym{\Gamma} \vec{G} (\vec{u}^j - \overline{\vec{m}})^\intercal (\vec{u}^j - \vec{u}^*) \\
	%&- \frac{\Lambda}{J} \sum_{j=1}^J \left\langle \vec{u}^j - \vec{u}^* , \vec{C}_{\vec{u}\vec{u}} \nabla_{\vec{u}}\mathcal{A} \left(\vec{u} - \vec{u}^* \right)|_{\vec{u}=\vec{u}^j} \right\rangle < 0.
	\end{align*}\qedhere
\end{proof}

The assumptions that $\vec{G}^\intercal \vecsym{\Gamma} \vec{G}$ and the empirical covariance $\vec{C}_{\vec{u}\vec{u}}$ are positive definite are strong and in general not satisfied. To deal with the positive definiteness of $\vec{C}_{\vec{u}\vec{u}}$ usually a constant or time--dependent inflation of the covariance is considered~\cite{SchillingsPreprint,ChadaShillingsWeissmann2019,TongMajdaStuart2015}.

\subsection{Continuous limit in the number of ensembles} \label{sec:meanfieldLinConstr}

Typically, the EnKF method is applied for a fixed and finite ensemble size. It is clear that the computational and memory cost of the method increases with the number of the ensembles, but there is a substantial gain in accuracy. The analysis of the method was also studied in the large ensemble limit, see e.g.~\cite{CarrilloVaes,DingLi2019,ernstetal2015,Stuart2019MFEnKF,HertyVisconti2019,kwiatkowskimandel2015,lawtembinetempone2016,leglandmonbettran2009}. In this limit a slow down of the ensemble spread has been observed. In this section, we derive the corresponding mean--field limit of the continuous time equation~\eqref{eq:limitDAElinearG} and provide an analysis of the resulting PDE equation.

We follow the classical formal derivation to formulate a mean--field equation of a particle system, see~\cite{CarrilloFornasierToscaniVecil2010,hatadmor2008,PareschiToscaniBOOK,Toscani2006}. Let us denote by
\begin{equation} \label{eq:kineticf-ulambda}
f = f(t,\vec{u},\vecsym{\lambda}) : \R^+ \times \R^d \times \R^m \to \R^+
\end{equation}
the compactly supported on $\R^d\times \R^m$ probability density of the pair $(\vec{u},\vecsym{\lambda})$ at time $t$ and introduce the following moments of $f$ at time $t$ with respect to $\vec{u}$ and $\vecsym{\lambda}$, respectively, as
\begin{equation} \label{eq:moments-ulambda}
\begin{bmatrix} \vec{m}_1(t) \\ \vecsym{\Lambda}_1(t) \end{bmatrix} = \iint_{\R^d\times\R^m} \begin{bmatrix} \vec{u} \\ \vecsym{\lambda} \end{bmatrix} f(t,\vec{u},\vecsym{\lambda}) \mathrm{d}\vec{u} \mathrm{d}\vecsym{\lambda}, \quad
\begin{bmatrix} \vec{m}_2(t) \\ \vecsym{\Lambda}_2(t) \end{bmatrix} = \iint_{\R^d\times\R^m} \begin{bmatrix} \vec{u} \otimes \vec{u} \\ \vecsym{\lambda} \otimes \vecsym{\lambda} \end{bmatrix} f(t,\vec{u},\vecsym{\lambda}) \mathrm{d}\vec{u} \mathrm{d}\vecsym{\lambda}
\end{equation}

Since $(\vec{u},\vecsym{\lambda})\in\R^d\times\R^m$, the empirical measure is given by
\begin{equation} \label{eq:empiricalf}
	f(t,\vec{u},\vecsym{\lambda}) = \frac{1}J \sum_{j=1}^J \delta(\vec{u}^j - \vec{u}) \delta(\vecsym{\lambda}^j - \vecsym{\lambda}).
\end{equation}
This formulation allows for a representation of the covariance operator~\eqref{eq:covariance} as 
\begin{align*}
	\vecsym{\C}(t) &= \iint_{\R^d\times\R^m} \vec{u} \otimes \vec{u} f(t,\vec{u},\vecsym{\lambda}) \mathrm{d}\vec{u} \mathrm{d}\vecsym{\lambda} - \iint_{\R^d\times\R^m} \vec{u} f(t,\vec{u},\vecsym{\lambda}) \mathrm{d}\vec{u} \mathrm{d}\vecsym{\lambda} \iint_{\R^d\times\R^m} \vec{u} f(t,\vec{u},\vecsym{\lambda}) \mathrm{d}\vec{u} \mathrm{d}\vecsym{\lambda} \\
	& = \vec{m}_2(t) - \vec{m}_1(t) \otimes \vec{m}_1(t) \geq 0.
\end{align*}
Let us denote $\varphi(\vec{u},\vecsym{\lambda}) \in C_0^\infty(\R^d\times\R^m)$ a test function. We compute
\begin{align*}
\frac{\mathrm{d}}{\mathrm{d}t} \left\langle f , \varphi \right\rangle &= \frac{\mathrm{d}}{\mathrm{d}t} \int_{\R^d} \frac{1}{J} \sum_{j=1}^J \delta(\vec{u} - \vec{u}^j) \delta(\vecsym{\lambda}^j - \vecsym{\lambda}) \varphi(\vec{u},\vecsym{\lambda}) \mathrm{d}\vec{u} \mathrm{d}\vecsym{\lambda} \\
&= - \frac{1}{J} \sum_{j=1}^J \nabla_\vec{u} \varphi(\vec{u}^j,\vecsym{\lambda}^j) \cdot \vecsym{\C} \left( \nabla_\vec{u} \Phi(\vec{u}^j,\vec{y}) + \vec{J}_\mathcal{A}^\intercal(\vec{u}^j) \vecsym{\lambda}^j \right) \\
&= - \iint_{\R^d\times\R^m} \nabla_\vec{u} \varphi(\vec{u},\vecsym{\lambda}) \cdot \vecsym{\C} \left( \nabla_\vec{u} \Phi(\vec{u},\vec{y}) + \vec{J}_\mathcal{A}^\intercal(\vec{u}) \vecsym{\lambda} \right) f(t,\vec{u}) \mathrm{d}\vec{u} \mathrm{d}\vecsym{\lambda}
\end{align*}
coupled to the mean--field of the algebraic constraint. Finally the mean--field kinetic equation corresponding to the DAE system~\eqref{eq:limitDAElinearG} reads:
\begin{equation} \label{eq:kineticWithConstr}
\begin{aligned}
\partial_t f(t,\vec{u},\vecsym{\lambda}) - \nabla_{\vec{u}} \cdot \left[ \vecsym{\C}(t) ( \nabla_{\vec{u}} \Phi(\vec{u},\vec{y}) + \vec{J}_\mathcal{A}^\intercal(\vec{u}) \vecsym{\lambda} ) f(t,\vec{u},\vecsym{\lambda}) \right] &= 0 \\
\iint_{\R^d\times\R^m} \mathcal{A}(\vec{u}) f(t,\vec{u},\vecsym{\lambda}) \mathrm{d}\vec{u} \mathrm{d}\vec{\lambda} &= \vec{0}_{\mathbb{R}^m}.
\end{aligned}
\end{equation}

\begin{proposition} \label{th:kineticSol}
	Let $f(t,\vec{u},\vecsym{\lambda})$ be a solution in distributional sense of the mean--field equation~\eqref{eq:kineticWithConstr} at $t>0$ for compactly supported initial probability distribution $f(t=0,\vec{u},\vecsym{\lambda})$. Then, $f(\vec{u},\vecsym{\lambda}) = \delta(\vec{u}-\vec{v}) \delta(\vecsym{\lambda}-\vecsym{\mu})$ is a steady state solution in distributional sense of~\eqref{eq:kineticWithConstr} provided that either $\vecsym{\C} = \vec{0}_{\mathbb{R}^{d\times d}}$ or $(\vec{v},\vecsym{\mu})$ is a KKT point of the minimization problem
	$$\min_{\vec{u}\in\R^d} \Phi(\vec{u},\vec{y}) \ \mbox{ subject to } \ \mathcal{A}(\vec{u}) = \vec{0}_{\R^m}.$$
\end{proposition}
\begin{proof}
	Let $\varphi(\vec{u},\vecsym{\lambda}) \in C_0^\infty(\R^d\times\R^m)$ be a test function. Then, weak steady state solutions, say $f^\infty(\vec{u},\vecsym{\lambda})$, to~\eqref{eq:kineticWithConstr} satisfy formally
	\begin{align*}
		\iint_{\R^d\times\R^m} \nabla_{\vec{u}} \varphi(\vec{u},\vecsym{\lambda}) \cdot \vecsym{\C} \left( \nabla_{\vec{u}} \Phi(\vec{u},\vec{y}) + \vec{J}_\mathcal{A}^\intercal(\vec{u}) \vecsym{\lambda} )  \right) \mathrm{d}f^\infty(\vec{u},\vecsym{\lambda}) &= 0 \\
		\iint_{\R^d\times\R^m} \varphi(\vec{u},\vecsym{\lambda}) \mathcal{A}(\vec{u}) \mathrm{d}f^\infty(\vec{u},\vecsym{\lambda}) &= \vec{0}_{\mathbb{R}^m}
	\end{align*}
	Substituting $f^\infty$ with $f(\vec{u},\vecsym{\lambda}) = \delta(\vec{u}-\vec{v}) \delta(\vecsym{\lambda}-\vecsym{\mu})$, we easily obtain that the above conditions are satisfied when either $\vecsym{\C} = \vec{0}_{\mathbb{R}^{d\times d}}$ or $\nabla_{\vec{u}} \Phi(\vec{u},\vec{y})|_{\vec{u}=\vec{v}} = - \left( \vec{J}_\mathcal{A}^\intercal(\vec{u}) \vecsym{\lambda} \right)|_{\vec{u}=\vec{v},\vecsym{\lambda}=\vecsym{\mu}}$ and simultaneously $\mathcal{A}(\vec{u})|_{\vec{u}=\vec{v}} = \vec{0}_{\mathbb{R}^m}$.
\end{proof}

The previous proposition states that, as in the discrete case, not all the steady states are in the kernel of $\vecsym{\C}$ due to the presence of the multipliers.

It is clear that the mean--field equation~\eqref{eq:kineticWithConstr} does not provide a closed differential system of the moments. In order to show an energy decay estimate, we perform a preliminary moment analysis in the simple setting of a one--dimensional control and scalar constraint. Thus, similarly to the previous section, we assume that $(u,\lambda) \in \R\times\R$ and consider a strictly convex quadratic constraint $\mathcal{A}(u) \in \R$. We write
$$
	\nabla_u \Phi(u,y) = b_1 u + b_2, \quad \nabla_u \mathcal{A}(u) = h_1 u + h_2, \ h_1>0.
$$
Then by the second equation in~\eqref{eq:kineticWithConstr} we have
$$
	0 = \iint_{\R^2} \left( \frac12 h_1 u^2 f(t,u,\lambda) + h_2 u f(t,u,\lambda) \right) \mathrm{d}u \mathrm{d}\lambda = \frac12 h_1 m_2(t) + h_2 m_1(t)
$$
which implies the following link between the first and the second moment due to the constraint:
\begin{equation} \label{eq:m2}
	m_2(t) = - 2 \frac{h_2}{h_1} m_1(t).
\end{equation}
Then, it suffices to study the evolution of the first moment in order to obtain information on the evolution of the second moment. Using the first equation in~\eqref{eq:kineticWithConstr} we derive the following evolution equation for the first moment:
\begin{equation} \label{eq:m1}
	\frac{\mathrm{d}}{\mathrm{d}t} m_1(t) = - C \iint_{\R^2} \nabla_u \Phi(u,y) f(t,u,\lambda) \mathrm{d}u \mathrm{d}\lambda - C \iint_{\R^2} \left( h_1 u \lambda f(t,u,\lambda) + h_2 \lambda f(t,u,\lambda) \right) \mathrm{d}u \mathrm{d}\lambda.
\end{equation}
The evolution of the multiplier is computed from~\eqref{eq:kineticWithConstr} obtaining for each $k\geq0$
$$
	\frac{\mathrm{d}}{\mathrm{d}t} \iint_{\R^2} \lambda^k f(t,u,\lambda) \mathrm{d}u \mathrm{d}\lambda = 0 \ \Rightarrow \ \iint_{\R^2} \lambda^k f(t,u,\lambda) \mathrm{d}u \mathrm{d}\lambda = \iint_{\R^2} \lambda^k f_0(u,\lambda) \mathrm{d}u \mathrm{d}\lambda =: \Lambda_k.
$$
For an arbitrary small positive quantity $\epsilon$ we have
$$
	\iint_{\R^2} (\sqrt{\epsilon} u) \frac{\lambda}{\sqrt{\epsilon}} f(t,u,\lambda) \mathrm{d}u \mathrm{d}\lambda \leq \frac\epsilon2 m_2(t) + \frac{1}{2\epsilon} \Lambda_2.
$$
Using the above relation and~\eqref{eq:m2}, from~\eqref{eq:m1} we obtain the following bound for the evolution equation of the first moment:
$$
	\frac{\mathrm{d}}{\mathrm{d}t} m_1(t) \leq L_\epsilon \C(t) \left( m_1(t) + \frac{K_\epsilon}{L_\epsilon} \right), \quad L_\epsilon= \epsilon h_2 - b_1, \quad K_\epsilon = -b_2 - \frac{h_1}{2\epsilon} \Lambda_2 - h_2 \Lambda_1
$$
with $L_\epsilon<0$ for $\epsilon$ sufficiently small. Defining $\tilde{m}_1(t) = m_1(t) + \frac{K_\epsilon}{L_\epsilon}$, by application of the Gronwall lemma we have
$$
	\frac{\mathrm{d}}{\mathrm{d}t} \tilde{m}_1(t) \leq L_\epsilon \C(t) \tilde{m}_1(t) \ \Rightarrow \ \tilde{m}_1(t) \leq \tilde{m}_1(0) \exp\left(L_\epsilon \int_0^t \C(\tau) \mathrm{d}\tau\right) \to 0.
$$
Finally,
$$
	m_1(t) \leq - \frac{K_\epsilon}{L_\epsilon} = \left| \frac{K_\epsilon}{L_\epsilon} \right|
$$
and the mean is hence bounded. Consequently, also the second moment $m_2$ is bounded by the relation~\eqref{eq:m2}. This shows that, according to Proposition~\ref{th:kineticSol}, not necessarily all the steady states are Dirac distributions on $\R^d\times\R^m$.

\section{Numerical experiments} \label{sec:numerics}

The simulations are performed for the case of linear models $\vec{G}$ by solving the system of DAEs~\eqref{eq:limitDAElinearG}. The details on the scheme and on the setting of the experiments are as follows. Many numerical methods for DAEs are known and based on Runge--Kutta and BDF methods, see e.g.~the monographs~\cite{BrenanBook1996,HairerWanner1996}. Here, for the sake of simplicity, we employ straightforwardly first order implicit--explicit (IMEX) time integration coupled to Newton's method. The IMEX  
discretization of~\eqref{eq:limitDAElinearG} with time step $\Delta t$ then reads as
\begin{equation} \label{eq:generalIMEX}
\begin{aligned}
	\vec{u}_j^{n+1} &= \vec{u}_j^{n} - \Delta t \vec{C}_{\vec{u}\vec{u}}^n \nabla_{\vec{u}} \Phi(\vec{u}_j^n,\vec{y}) - \Delta t \vec{C}_{\vec{u}\vec{u}}^{n} \vec{J}_\mathcal{A}(\vec{u}_j^{n+1}) \vecsym{\lambda}_j^{n+1} \\
	\vec{0}_{\R^m} &= \mathcal{A}(\vec{u}_j^{n+1}).
\end{aligned}
\end{equation}
This represents a possible high dimensional system of nonlinear equations for $\vec{u}_j^{n+1} \in \R^d$ and $\vecsym{\lambda}_j^{n+1} \in \R^m$, for $j=1,\dots,J$, to be solved at each time integration step $n$. Inspired by the derivation of~\eqref{eq:limitDAElinearG} we consider the covariance matrix $\vec{C}_{\vec{u}\vec{u}}$ computed explicitly at time level $n$. This is motivated by the fact that $\vec{C}_{\vec{u}\vec{u}}$ can be written in terms of moments. Observe that then the IMEX discretization is~\eqref{eq:limitDAElinearG} in the limit $\Delta t \to 0^+$. 

The differential point of view allows us to employ an adaptive time step to avoid stability issues. In the numerical experiments we use
$$
	\Delta t \leq \frac{1}{\max_{i}\left(|(\Re(\mu_i)|\right)}
$$
where the $\mu_i$'s are the eigenvalues of $\vec{C}_{\vec{u}\vec{u}}^n \vec{G}^\intercal \vecsym{\Gamma} \vec{G}$. As we observe that $\vec{C}_{\vec{u}\vec{u}}^n \vec{G}^\intercal \vecsym{\Gamma} \vec{G}$ is characterized by large spectral radius at initial time that reduces over time, the adaptive computation of $\Delta t$ allows also to reach equilibrium in less time steps compared to the choice of a fixed $\Delta t$.

We are interested to solve the inverse problem described in Section~\ref{sec:linearEqConstr} aimed to find the force function of the linear elliptic equation in~\eqref{eq:ellipticEq}. This is a typical example used in the mathematical literature to test the property of the ensemble Kalman filter~\cite{Stuart2019CEnKF,HertyVisconti2019,iglesiaslawstuart2013,schillingsstuart2017}. We consider the same setup as in Section~\ref{sec:linearEqConstr}, but in this case we take into account a nonlinear constraint. In particular, we focus on a scalar ($m=1$) quadratic constraint $\mathcal{A} : \R^d \to \R$ of the form $\mathcal{A}(\vec{u}) = \frac12 \vec{u}^\intercal \vec{A} \vec{u} + \vec{u}^\intercal \vec{b}$, with $\vec{A}\in\R^{d\times d}$ and $\vec{b}\in\R^d$ given. In this setting, the gradient of $\mathcal{A}$ is linear, $\nabla_{\vec{u}} \mathcal{A} = \vec{A} \vec{u} + \vec{b}$. Therefore, \eqref{eq:generalIMEX} can be explicitly solved by $\vec{u}_j^{n+1}$ obtaining
\begin{equation} \label{eq:collectImplicitTerm}
	\vec{u}_j^{n+1} = \left( \vec{I} + \Delta t \lambda_j^{n+1} \vec{C}_{\vec{u}\vec{u}}^n \right)^{-1} \left( \vec{u}_j^n - \Delta t \vec{C}_{\vec{u}\vec{u}} \nabla_{\vec{u}} \Phi(\vec{u}_j^n,\vec{y}) - \Delta t \lambda_j^{n+1} \vec{C}_{\vec{u}\vec{u}} \vec{b} \right),
\end{equation}
where $\vec{I}\in\R^{d\times d}$ is the identity matrix.
Then, the scalar multiplier $\lambda_j^{n+1}$ is computed by solving the nonlinear equation $\mathcal{A}(\vec{u}_j^{n+1})=0$ for each $j=1,\dots,J$ at each integration step. The multipliers are finally inserted into~\eqref{eq:collectImplicitTerm} to determine the final update of the feasible ensembles at time level $n+1$.

In each example the exact force function is chosen as $u(x) = \sin(\pi x)$ and then suitable modified in order to satisfy the constraint. Also, the initial condition $\vec{u}_j(0) = \vec{u}_j^0$ on the ensembles satisfies the constraints. The vector $\vec{b}$, characterizing the linear term in the constraint, is always randomly sampled from a Gaussian distribution with zero mean and standard deviation $\frac12$.

Information on the simulation results is presented in the following norms
\begin{equation} \label{eq:normErr}
E(t) = \frac{1}{J} \sum_{j=1}^J \| \vec{e}^j(t) \|^2, \quad R(t) = \frac{1}{J} \sum_{j=1}^J \| \vec{r}^j(t) \|^2
\end{equation}
at each iteration, where the quantities $\vec{e}^j(t)$ and $\vec{r}^j(t)$ are the spread and the residual of the ensembles, respectively, as defined by~\eqref{eq:spreadresidual}. The residual is measured by taking $\vec{u}^*$ as the truth solution $\vec{u}^\dagger$.

Another additional important quantities is given by the misfit which allows to measure the quality of the solution at each iteration. The misfit for the $j$-th sample is defined as
\begin{equation} \label{eq:singleMisfit}
\vecsym{\vartheta}^j(t) = \vec{G} \vec{r}^j(t) - \vecsym{\eta},
\end{equation}
where $\vec{G}$ is again the finite difference discretization of the continuous linear operator defining the elliptic PDE~\eqref{eq:ellipticEq}.
By using~\eqref{eq:singleMisfit} we finally look at
\begin{equation} \label{eq:misfit}
\vartheta(t) = \frac{1}{J} \sum_{j=1}^J \| \vecsym{\vartheta}^j(t) \|^2.
\end{equation}
Driving this quantity to zero leads to over--fitting of the solution. For this reason, usually it is suitable introducing a stopping criterion which avoids this effect. In the following we will consider the discrepancy principle which check and stop the simulation if the condition $\vartheta \leq \| \vecsym{\eta} \|^2$ is satisfied, with $\vecsym{\eta}$ measurement noise as described in the problem setup in Section~\ref{sec:linearEqConstr}.

\subsection{Quadratic convex constraint} \label{sec:convexExample}

The first situation we consider is the case of a strictly convex constraint by taking $\vec{A} = \vec{I} \in \R^{d\times d}$, i.e.,
$$
	\mathcal{A}(\vec{u}) = \sum_{k=1}^d \vec{u}_{(k)}^2 + \vec{b}_{(k)} \vec{u}_{(k)}.
$$
In order to allow the force function $u(x)$ to satisfy the constraint, we artificially modify it on the grid introduced on the domain $[0,\pi]$. In particular, we determine $\{ \vec{u}_{(k)} \}_{k=1}^{\frac{d}{2}}$ in such a way $\vec{u}_{(k)}^2 + \vec{b}_{(k)} \vec{u}_{(k)} = - \vec{u}_{(d-k+1)}^2 - \vec{b}_{(d-k+1)} \vec{u}_{(d-k+1)}$, for $k=1,\dots,\frac{d}{2}$. The initial ensembles are sampled from a multivariate normal distribution and then also artificially modified to satisfy the constraint. We consider the same setup as in Section~\ref{sec:linearEqConstr}, thus a noise level $\gamma = 0.01$, $d=K=2^8$.

\begin{figure}[t!]
	\centering
	\includegraphics[width=\textwidth]{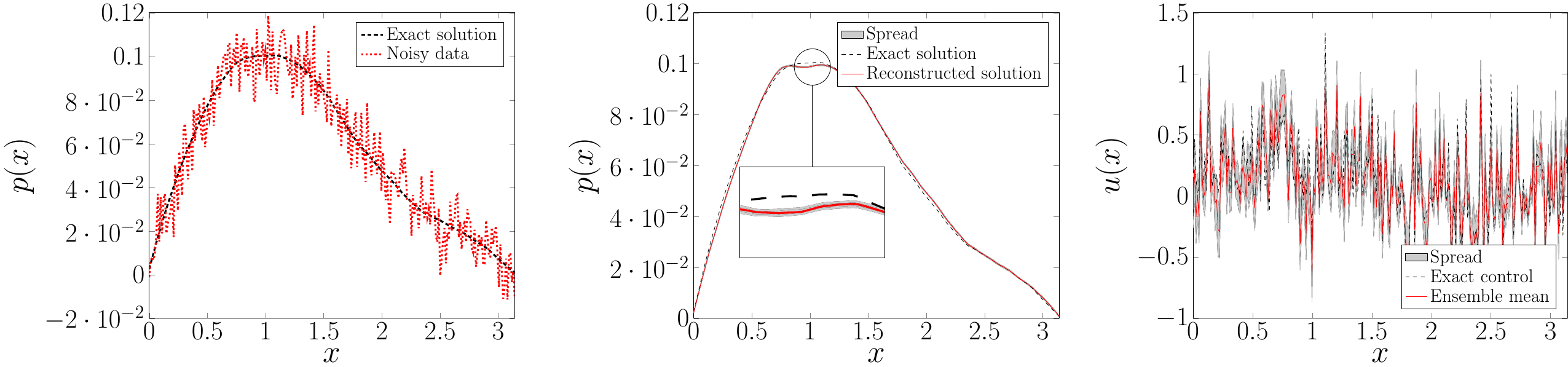}
	\caption{Solution provided by the constrained ensemble Kalman filter using $J=160$ ensembles on the strictly convex nonlinear constraint.\label{fig:convexCEnKF}}
\end{figure}

In Figure~\ref{fig:convexCEnKF} we compute the solution obtained with $J=160$ ensembles. The left plot shows the solution $p(x)$ of the PDE~\eqref{eq:ellipticEq} and noisy observations. The reconstruction of $p(x)$ is provided in the center panel and it is obtained by using the mean of the ensembles represented in the right panel. The gray areas gives information on the spread due to the ensembles. The constrained ensemble Kalman filter accurately reconstructs both the control and its projection through the PDE model. 

\begin{figure}[t!]
	\centering
	\includegraphics[width=\textwidth]{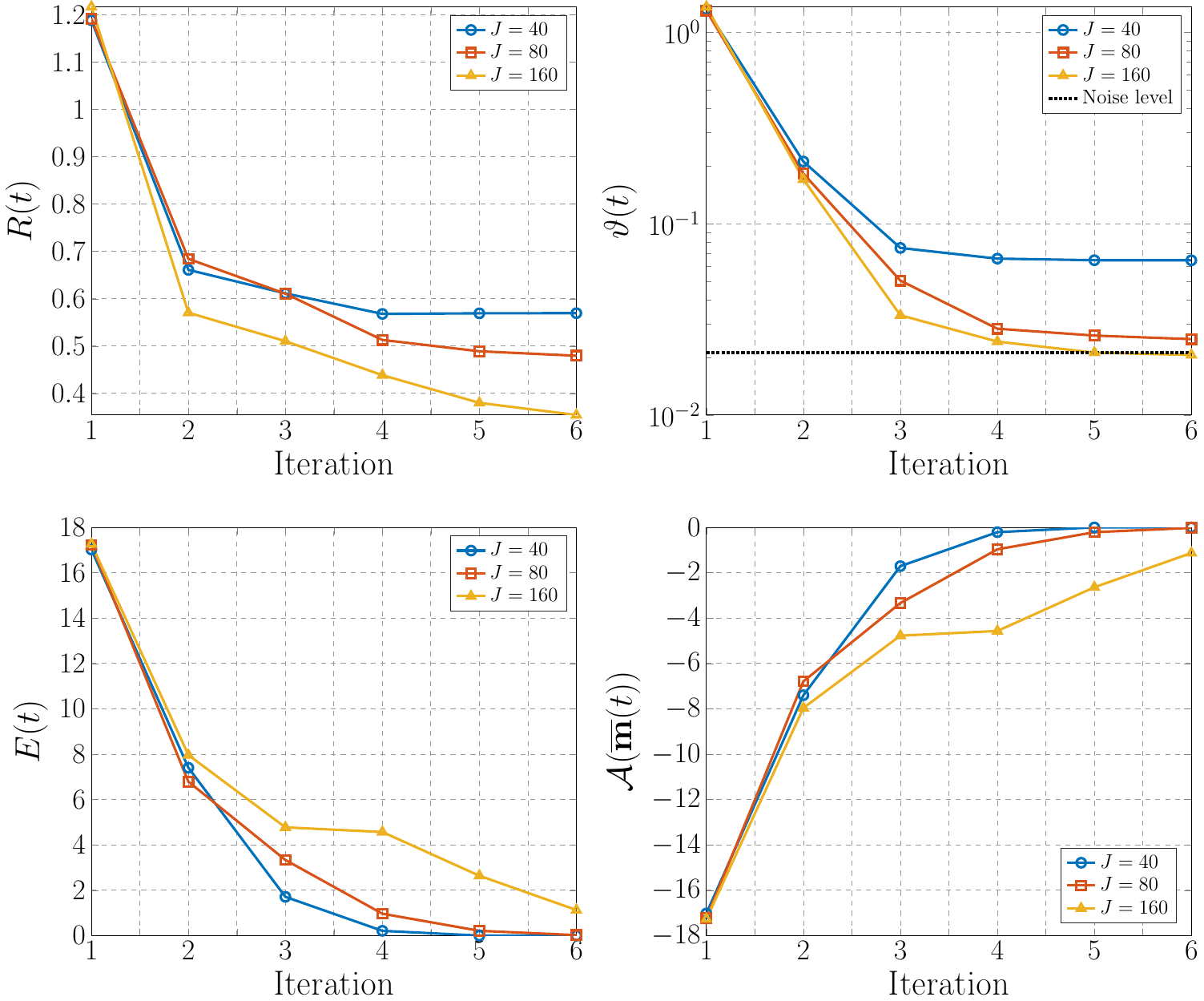}
	\caption{Time evolution of the residual (top left), the misfit (top right), the ensemble collapse (bottom left) and the value of the constraint computed on the ensemble mean (bottom right) for the strictly convex nonlinear constraint.\label{fig:convexAnalysis}}
\end{figure}

The analysis of the solution is performed with three values of the ensemble size $J\in\{40,80,160\}$, see Figure~\ref{fig:convexAnalysis}. In particular, we consider the behavior in time of the residual, the misfit, the ensemble collapse and the value of the constraint computed on the ensemble mean. We observe that increasing the value of the ensembles allows a more accurate reconstruction. In fact, the solution is able to meet the discrepancy principle, i.e.~the misfit reaches the noise level, and the residual values decreases. However, the collapse to the mean slows down causing the increase of the ensemble spread. This in turn does not allow the mean to satisfy exactly the constraint when the simulation stops due to the discrepancy principle.

\begin{figure}[t!]
	\centering
	\includegraphics[width=\textwidth]{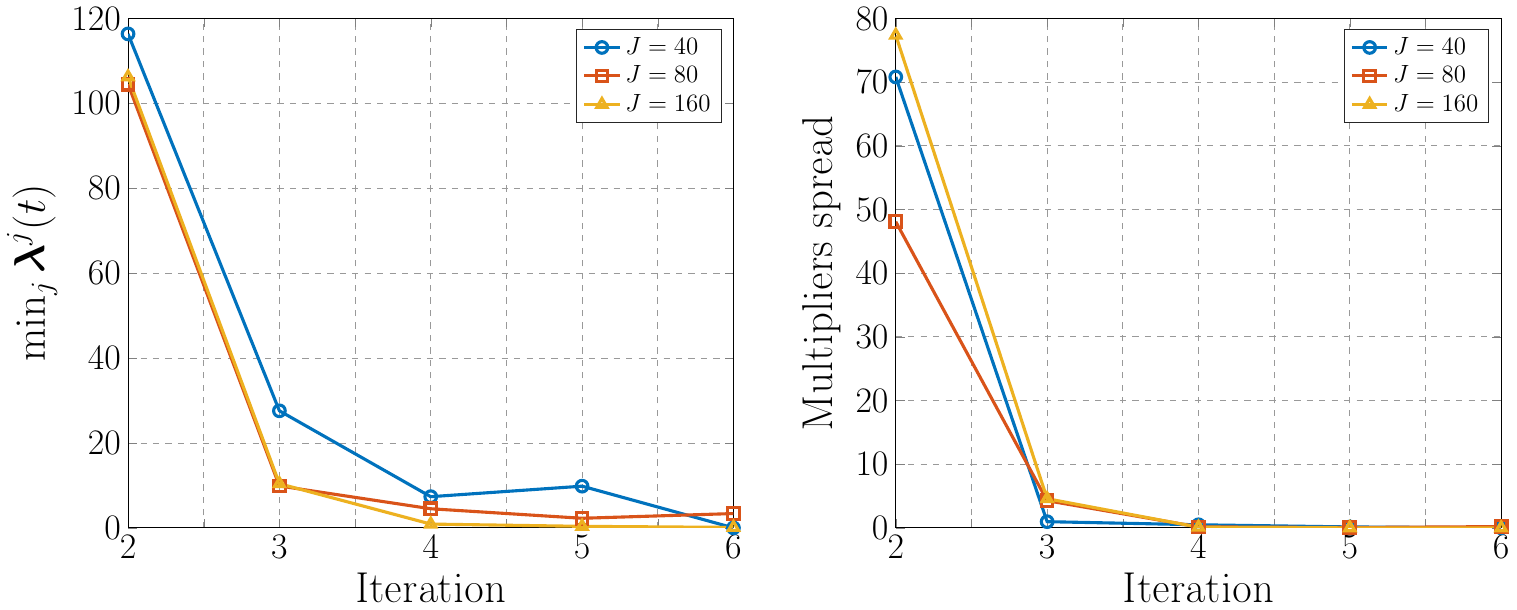}
	\caption{Time evolution of the minimum value of the multipliers (left) and of the multipliers spread around their mean (right) for the strictly convex nonlinear constraint.\label{fig:convexMultipliersAnalysis}}
\end{figure}

In Figure~\ref{fig:convexMultipliersAnalysis} we show the evolution of the minimum value of the multipliers, noticing that it remains always positive. Moreover, the spread of the multipliers to their mean decays to zero. This two results justify the ensemble spread observed in Figure~\ref{fig:convexAnalysis} and reflect the analysis provided in Section~\ref{sec:linearModel}, cf. the discussion after Proposition~\ref{th:meanKKT} and the statement of Proposition~\ref{th:spread}

\subsection{Quadratic non--convex constraint}

The second situation we consider is the case of a non--convex constraint by taking $\vec{A} = \begin{bmatrix} \vec{I} & \vec{0} \\ \vec{0} & -\vec{I} \end{bmatrix} \in \R^{d\times d}$. Then the constraint reads as
$$
\mathcal{A}(\vec{u}) = \sum_{k=1}^{\frac{d}{2}} \vec{u}_{(k)}^2 + \vec{b}_{(k)} \vec{u}_{(k)} + \sum_{\ell=\frac{d}{2}+1}^d \vec{b}_{(k)} \vec{u}_{(k)} - \vec{u}_{(\ell)}^2.
$$
We use the same setup as in Section~\ref{sec:convexExample}. Now the force function $u(x)$ is artificially modified to satisfy the constraint by determining $\{ \vec{u}_{(k)} \}_{k=1}^{\frac{d}{2}}$ in such a way $\vec{u}_{(k)}^2 + \vec{b}_{(k)} \vec{u}_{(k)} = \vec{u}_{(d-k+1)}^2 - \vec{b}_{(d-k+1)} \vec{u}_{(d-k+1)}$, for $k=1,\dots,\frac{d}{2}$. We proceed similarly for the initial ensembles.

\begin{figure}[t!]
	\centering
	\includegraphics[width=\textwidth]{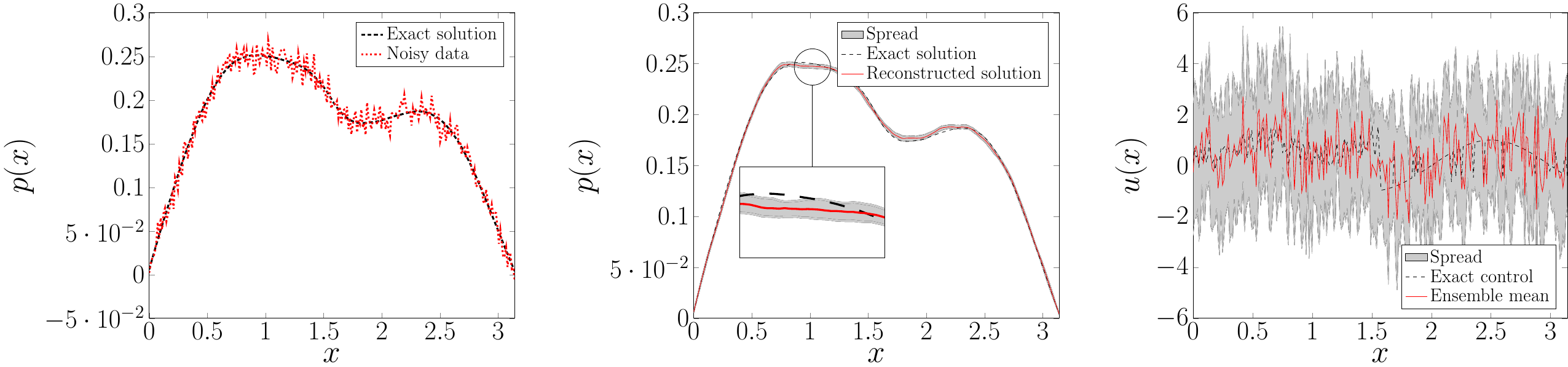}
	\caption{Solution provided by the constrained ensemble Kalman filter using $J=160$ ensembles on the non--convex nonlinear constraint.\label{fig:nonconvexCEnKF}}
\end{figure}

In Figure~\ref{fig:nonconvexCEnKF} we compute the solution obtained with $J=160$ ensembles. The left plot shows the solution $p(x)$ of the PDE~\eqref{eq:ellipticEq} and the noisy observations. The reconstruction of $p(x)$ is provided in the center panel and it is obtained by using the mean of the ensembles represented in the right panel. The gray areas gives information on the spread due to the ensembles. Compared to the previous example, here we notice that the ensembles are widely spread around the exact control function, cf. the right plot in Figure~\ref{fig:nonconvexCEnKF}. The method does not reproduce a very accurate control function but the application of the model still allows to obtain a good reconstruction of the solution.

\begin{figure}[t!]
	\centering
	\includegraphics[width=\textwidth]{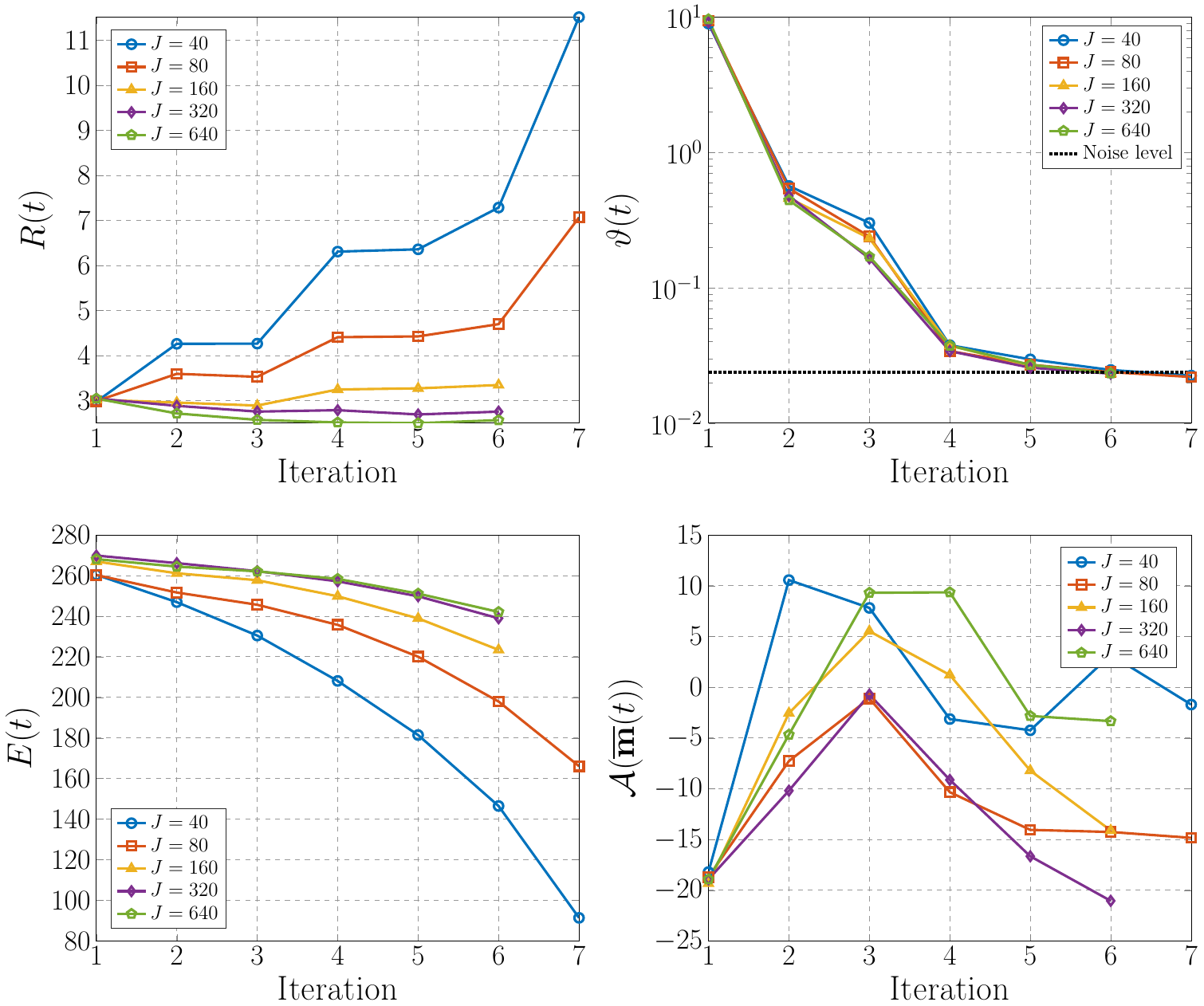}
	\caption{Time evolution of the residual (top left), the misfit (top right), the ensemble collapse (bottom left) and the value of the constraint computed on the ensemble mean (bottom right) for the non--convex nonlinear constraint.\label{fig:nonconvexAnalysis}}
\end{figure}

We perform the analysis of the method with five values of the ensemble size $J\in\{40,80,160,320,640\}$, see Figure~\ref{fig:nonconvexAnalysis}. Again we consider the behavior in time of the residual, the misfit, the ensemble collapse and the value of the constraint computed on the ensemble mean. We observe that the ensemble spread is decreasing in time very slowly, even with a small number of ensembles. This result also affects the value of the constraint computed on the mean of the ensembles, which is far from zero. The residual shows a decreasing behavior in time only with very large number of ensembles. However, in all cases the misfit is able to meet the noise level and therefore the discrepancy principle holds. These results can be motivated by the use of a non--convex constraint, %. In fact, assuming a non-convex constraint we cannot guarantee that the conditions for optimality are sufficient, see Proposition~\ref{th:ensembleKKT} and Proposition~\ref{th:meanKKT}. Namely,
and the method is providing a solution which minimizes the least square functional but the mean does not exactly satisfy the constraint. 

\section{Summary and perspectives} \label{sec:conclusion}

In this paper, inspired by~\cite{Stuart2019CEnKF}, we have focused on the formulation of the ensemble Kalman filter to solve constrained inverse problems. We have worked in the setting of optimization theory deriving first order necessary optimality conditions for the case of equality constraints in the space of the controls. We have observed that the method relaxes to the unconstrained one when linear equality constraints are considered. Therefore, we have mainly focused on nonlinear constraints and analyzed the method by computing continuous limits, in time and in the regime of infinitely many ensembles. The numerical results have shown that the method is able to provide solution to constrained inverse problems with quadratic convex and non--convex equality constraints.

We recall that in order to set up the mathematical formulation of the problem studied in this work, we have considered finite dimensional Hilbert spaces and equality constraints. As future perspectives it would be of interest to extend this study to arbitrary and possibly infinite dimensional Hilbert spaces, and to consider also inequality constraints.

\subsection*{Acknowledgments}
This research is funded by the Deutsche Forschungsgemeinschaft (DFG, German Research Foundation) under Germany's Excellence Strategy -- EXC-2023 Internet of Production -- 390621612 and supported also by DFG HE5386/15.

\appendix

\section{Detailed computations in the proof of Proposition~\ref{th:optimality}} \label{app:proof}

	\begin{itemize}
		\item Woodbury matrix identity:
		$$
		(\vec{M} + \vec{U} \vec{N} \vec{V})^{-1} = \vec{M}^{-1} - \vec{M}^{-1} \vec{U} (\vec{N}^{-1} + \vec{V} \vec{M}^{-1} \vec{U})^{-1} \vec{V} \vec{M}^{-1},
		$$
		for all matrices $\vec{M}$, $\vec{U}$, $\vec{N}$, $\vec{V}$.
		\item Application of the Woodbury matrix identity to $(\vec{H}^\intercal\vecsym{\Gamma}\vec{H}+\vec{C}_\varepsilon^{{n+1}^{-1}})^{-1}$:
		$$
		(\vec{H}^\intercal\vecsym{\Gamma}\vec{H}+\vec{C}_\varepsilon^{{n+1}^{-1}})^{-1} = \vec{C}_\varepsilon^{n+1} - \vec{C}_\varepsilon^{n+1} \vec{H}^\intercal (\vecsym{\Gamma}^{-1} + \vec{H}\vec{C}_\varepsilon^{n+1}\vec{H}^\intercal)^{-1}\vec{H}\vec{C}_\varepsilon^{n+1}
		$$
		\item From equation~\eqref{eq:vSol} to equation~\eqref{eq:vSol2}:
		\begin{align*}
		\vec{v}^{j,n+1} =& (\vec{H}^\intercal \vecsym{\Gamma} \vec{H} + \vec{C}_\varepsilon^{{n+1}^{-1}})^{-1} (\vec{H}^\intercal \vecsym{\Gamma} \vec{y}^{n+1} + \vec{C}_\varepsilon^{{n+1}^{-1}} \hat{\vec{v}}^{j,n+1}) \\ &- (\vec{H}^\intercal \vecsym{\Gamma} \vec{H} + \vec{C}_\varepsilon^{{n+1}^{-1}})^{-1} \vec{J}_{\tilde{\mathcal{A}}}^\intercal(\vec{v}^{j,n+1}) \vecsym{\lambda}^{j,n+1} \\
		=& (\vec{H}^\intercal \vecsym{\Gamma} \vec{H} + \vec{C}_\varepsilon^{{n+1}^{-1}})^{-1} \vec{C}_\varepsilon^{{n+1}^{-1}} \hat{\vec{v}}^{j,n+1} + (\vec{H}^\intercal \vecsym{\Gamma} \vec{H} + \vec{C}_\varepsilon^{{n+1}^{-1}})^{-1} \vec{H}^\intercal \vecsym{\Gamma} \vec{y}^{n+1} \\ &- (\vec{H}^\intercal \vecsym{\Gamma} \vec{H} + \vec{C}_\varepsilon^{{n+1}^{-1}})^{-1} \vec{J}_{\tilde{\mathcal{A}}}^\intercal(\vec{v}^{j,n+1}) \vecsym{\lambda}^{j,n+1} \\
		=& (\vec{C}_\varepsilon^{n+1} - \vec{C}_\varepsilon^{n+1} \vec{H}^\intercal (\vecsym{\Gamma}^{-1} + \vec{H}\vec{C}_\varepsilon^{n+1}\vec{H}^\intercal)^{-1}\vec{H}\vec{C}_\varepsilon^{n+1}) \vec{C}_\varepsilon^{{n+1}^{-1}} \hat{\vec{v}}^{j,n+1} \\
		&+ \vec{C}_\varepsilon^{n+1} \vec{H}^\intercal (\vec{H}\vec{C}_\varepsilon^{n+1}\vec{H}^\intercal+\vecsym{\Gamma}^{-1})^{-1} \vec{y}^{n+1} \\
		&-(\vec{C}_\varepsilon^{n+1} - \vec{C}_\varepsilon^{n+1} \vec{H}^\intercal (\vecsym{\Gamma}^{-1} + \vec{H}\vec{C}_\varepsilon^{n+1}\vec{H}^\intercal)^{-1}\vec{H}\vec{C}_\varepsilon^{n+1}) \vec{J}_{\tilde{\mathcal{A}}}^\intercal(\vec{v}^{j,n+1}) \vecsym{\lambda}^{j,n+1} \\
		=& \hat{\vec{v}}^{j,n+1} - \vec{C}_\varepsilon^{n+1} \vec{H}^\intercal (\vecsym{\Gamma}^{-1} + \vec{H}\vec{C}_\varepsilon^{n+1}\vec{H}^\intercal)^{-1}\vec{H} \hat{\vec{v}}^{j,n+1} \\
		&+ \vec{C}_\varepsilon^{n+1} \vec{H}^\intercal (\vec{H}\vec{C}_\varepsilon^{n+1}\vec{H}^\intercal+\vecsym{\Gamma}^{-1})^{-1} \vec{y}^{n+1} \\
		&-(\vec{C}_\varepsilon^{n+1} - \vec{C}_\varepsilon^{n+1} \vec{H}^\intercal (\vecsym{\Gamma}^{-1} + \vec{H}\vec{C}_\varepsilon^{n+1}\vec{H}^\intercal)^{-1}\vec{H}\vec{C}_\varepsilon^{n+1}) \vec{J}_{\tilde{\mathcal{A}}}^\intercal(\vec{v}^{j,n+1}) \vecsym{\lambda}^{j,n+1} \\
		=& \hat{\vec{v}}^{j,n+1} + \vec{C}_\varepsilon^{n+1} \vec{H}^\intercal (\vec{H} \vec{C}_\varepsilon^{n+1} \vec{H}^\intercal + \vecsym{\Gamma}^{-1})^{-1} (\vec{y}^{n+1} - \vec{H}\hat{\vec{v}}^{j,n+1}) \\
		&- (\vec{C}_\varepsilon^{n+1} - \vec{C}_\varepsilon^{n+1} \vec{H}^\intercal (\vec{H} \vec{C}_\varepsilon^{n+1} \vec{H}^\intercal + \vecsym{\Gamma}^{-1})^{-1} \vec{H} \vec{C}_\varepsilon^{n+1} ) \vec{J}_{\tilde{\mathcal{A}}}^\intercal(\vec{v}^{j,n+1}) \vecsym{\lambda}^{j,n+1}
		\end{align*}
		\item From equation~\eqref{eq:vSol2} to the constrained EnKF updated formula~\eqref{eq:updateControl}.\\
		We multiply equation~\eqref{eq:vSol2} by the observation matrix $\vec{H}^\perp$
		\begin{align*}
		\vec{u}^{j,n+1} = \vec{H}^\perp \vec{v}^{j,n+1} =& \vec{H}^\perp \hat{\vec{v}}^{j,n+1} + \vec{H}^\perp \vec{C}_\varepsilon^{n+1} \vec{H}^\intercal (\vec{H} \vec{C}_\varepsilon^{n+1} \vec{H}^\intercal + \vecsym{\Gamma}^{-1})^{-1} (\vec{y}^{n+1} - \vec{H}\hat{\vec{v}}^{j,n+1}) \\
		&- \vec{H}^\perp (\vec{C}_\varepsilon^{n+1} - \vec{C}_\varepsilon^{n+1} \vec{H}^\intercal (\vec{H} \vec{C}_\varepsilon^{n+1} \vec{H}^\intercal + \vecsym{\Gamma}^{-1})^{-1} \vec{H} \vec{C}_\varepsilon^{n+1} ) \vec{J}_{\tilde{\mathcal{A}}}^\intercal(\vec{v}^{j,n+1}) \vecsym{\lambda}^{j,n+1}.
		\end{align*}
		and, since the map $\vec{M}\to\vec{M}^{-1}$ is continuous over the set of invertible matrices, letting $\varepsilon\to 0$
		\begin{align*}
		\vec{u}^{j,n+1} =& \vec{H}^\perp \hat{\vec{v}}^{j,n+1} + \vec{H}^\perp \vec{C}^{n+1} \vec{H}^\intercal (\vec{H} \vec{C}^{n+1} \vec{H}^\intercal + \vecsym{\Gamma}^{-1})^{-1} (\vec{y}^{n+1} - \vec{H}\hat{\vec{v}}^{j,n+1}) \\
		&- \vec{H}^\perp (\vec{C}^{n+1} - \vec{C}^{n+1} \vec{H}^\intercal (\vec{H} \vec{C}^{n+1} \vec{H}^\intercal + \vecsym{\Gamma}^{-1})^{-1} \vec{H} \vec{C}^{n+1} ) \vec{J}_{\tilde{\mathcal{A}}}^\intercal(\vec{v}^{j,n+1}) \vecsym{\lambda}^{j,n+1}.
		\end{align*}
		Noticing that
		\begin{gather*}
		\vec{H}^\perp \hat{\vec{v}}^{j,n+1} = \vec{u}^{j,n}, \quad
		\vec{H}^\perp \vec{C}^{n+1} \vec{H}^\intercal = \vec{C}^{n+1}_{\vec{u}\vec{w}}, \quad %\\
		\vec{H} \vec{C}^{n+1} \vec{H}^\intercal = \vec{C}^{n+1}_{\vec{w}\vec{w}}, \quad \vec{H}\hat{\vec{v}}^{j,n+1} = \G(\vec{u}^{j,n}) \\
		\vec{H} \vec{C}^{n+1} \vec{J}_{\tilde{\mathcal{A}}}^\intercal(\vec{v}^{j,n+1}) = \vec{C}^{{n+1}^\intercal}_{\vec{u}\vec{w}} \vec{J}_{\mathcal{A}}^\intercal(\vec{u}^{j,n+1}), \quad \vec{H}^\perp \vec{C}^{n+1} \vec{J}_{\tilde{\mathcal{A}}}^\intercal(\vec{v}^{j,n+1}) = \vec{C}^{n+1}_{\vec{u}\vec{u}} \vec{J}_{\mathcal{A}}^\intercal(\vec{u}^{j,n+1})
		\end{gather*}
		we finally obtain
		\begin{align*}
		\vec{u}^{j,n+1} =& \vec{u}^{j,n} + \vec{C}^{n+1}_{\vec{u}\vec{w}} ( \vec{C}^{n+1}_{\vec{w}\vec{w}} + \vecsym{\Gamma}^{-1} )^{-1} (\vec{y}^{n+1} - \G(\vec{u}^{j,n}) )\\
		&+ \vec{C}^{n+1}_{\vec{u}\vec{w}} ( \vec{C}^{n+1}_{\vec{w}\vec{w}} + \vecsym{\Gamma}^{-1} )^{-1} \vec{C}^{{n+1}^\intercal}_{\vec{u}\vec{w}} \vec{J}_\mathcal{A}^\intercal(\vec{u}^{j,n+1}) \vecsym{\lambda}^{j,n+1} - \vec{C}^{n+1}_{\vec{u}\vec{u}} \vec{J}_\mathcal{A}^\intercal(\vec{u}^{j,n+1}) \vecsym{\lambda}^{j,n+1}.
		\end{align*}
	\end{itemize}

\section{List of most used symbols} \label{app:symbols}

\begin{tabular}{@{}ll@{}} 
	\textbf{Symbol} & \textbf{Description} \\[1ex]
	$\R, \ \mathbb{N}$ & spaces of real numbers and positive integers \\[0.5ex]
	$\vec{y}\in\R^K, \ K\in\mathbb{N}$ & observations or measurements \\[0.5ex]
	$\vec{u}\in\R^d, \ d\in\mathbb{N}$ & unknown control \\[0.5ex]
	$\G:\R^d \to \R^K$ & forward model \\[0.5ex]
	$\vec{G}:\R^d \to \R^K$ & forward linear model \\[0.5ex]
	$\vecsym{\Gamma}^{-1} \in\R^{K\times K}$ & covariance matrix of the noise \\[0.5ex]
	$\vecsym{\eta}\sim\mathcal{N}(\vec{0},\vecsym{\Gamma}^{-1})$ & normally distributed measurement noise \\[0.5ex]
	$\Phi(\vec{u},\vec{y})=\frac12 \left\| \vecsym{\Gamma}^{\frac12} (\vec{y} - \G(\vec{u})) \right\|^2$ & least square functional \\[0.5ex]
	$\mathcal{A}:\R^d\to\R^m, \ m\in\mathbb{N}$ & constraint functions \\[0.5ex]
	$\tilde{\mathcal{A}}:\R^{d+K}\to\R^m$ & extension of the constraint functions to the space $\R^{d+K}$ \\[0.5ex]
	$\vec{A}\in\R^{m\times d}$ & linear constraint functions \\[0.5ex]
	$\vec{J}_\mathcal{A}\in\R^{m\times d}$ & Jacobian matrix of the constraints \\[0.5ex]
	$\vec{u}^j\in\R^d, \ j\in\{1,\dots,J\}$ & $j$--th ensemble member \\[0.5ex]
	$\vecsym{\lambda}^j\in\R^,m, \ j\in\{1,\dots,J\}$ & multiplier related to the $j$--th ensemble member\\[0.5ex]
	$\vec{C}_{\vec{u}\vec{u}} \in\R^{d\times d}$ & empirical covariance of the ensemble \\[0.5ex]
	$\vec{C}_{\vec{u}\vec{w}} \in\R^{d\times K}$ & empirical covariance of the ensemble and its image with respect\\ & the forward model \\[0.5ex]
	$\vec{C}_{\vec{w}\vec{w}} \in\R^{K\times K}$ & empirical covariance of the image of the ensemble with respect\\ & the forward model \\[0.5ex]
	$\overline{\vec{m}}\in\R^d$ & first moment (or mean) of the ensemble \\[0.5ex]
	$\doverline{\vec{m}}\in\R^{d\times d}$ & second moment of the ensemble \\[0.5ex]
	$\vec{e}^j\in\R^{d}$ & $j$--th ensemble spread to the mean \\[0.5ex]
	$\vec{r}^j\in\R^{d}$ & $j$--th ensemble residual \\[0.5ex]
	$\vecsym{\vartheta}^j\in\R^{K}$ & $j$--th ensemble misfit \\[0.5ex]
\end{tabular}

\bibliographystyle{plain}
\bibliography{references}

\end{document}